\documentclass[onefignum,onetabnum]{siamart190516}

\usepackage[caption=false]{subfig}

\usepackage{algpseudocode}
\algrenewcommand\algorithmicrequire{\textbf{Input:}}
\algrenewcommand\algorithmicensure{\textbf{Output:}}

\newsiamremark{example}{Example}

\usepackage{listings}
\definecolor{mauve}{rgb}{0.58,0,0.82}
\definecolor{dkgreen}{rgb}{0,0.6,0}
\usepackage{amsmath}
\usepackage{amssymb}
\usepackage{mathtools}
\newsiamremark{remark}{Remark}
\usepackage[caption=false]{subfig}
\usepackage{xcolor}
\definecolor{HUblue}{rgb}{0,0.2157,0.4235}
\definecolor{HUred}{rgb}{0.5412,0.0588,0.0784}
\definecolor{HUsand}{rgb}{0.8235,0.7529,0.4039}
\definecolor{HUgreen}{rgb}{0,0.3412,0.1725}
\usepackage{tikz}
\usetikzlibrary{%
    calc,%
    cd,%
    shapes.geometric%
}
\usepackage{tikz-3dplot}
\usepackage{pgfplots}
\pgfplotsset{%
    compat=newest,%
    every axis/.style={scale only axis},%
    grid style={densely dotted, semithick},%
}
\newcommand\drawslopetriangle[4][ST]{
    \pgfplotsextra
    {
        \pgfkeys{/pgf/fpu=true}
        \pgfmathsetmacro\leftcoord{#3};
        \pgfmathsetmacro\rightcoord{10*#3};
        \pgfmathsetmacro\bottomcoord{#4};
        \pgfmathsetmacro\topcoord{10^(#2)*#4};
        \pgfkeys{/pgf/fpu=false}
        \coordinate (#1-BL) at (axis cs:\leftcoord,\bottomcoord);
        \coordinate (#1-BR) at (axis cs:\rightcoord,\bottomcoord);
        \coordinate (#1-TL) at (axis cs:\leftcoord,\topcoord);
        \shadedraw[%
            bottom color = black!20,%
            middle color = black!5,%
            top color    = white,%
            draw         = black,%
            font         = \footnotesize%
        ]
        (#1-TL) -- (#1-BL) node[midway, left] {\(#2\)}
        -- (#1-BR) node[midway, below] {\(1\)} -- (#1-TL);
    }
}
\newcommand\drawswappedslopetriangle[4][SST]{
    \pgfplotsextra
    {
        \pgfkeys{/pgf/fpu=true}
        \pgfmathsetmacro\leftcoord{#3/10};
        \pgfmathsetmacro\rightcoord{#3};
        \pgfmathsetmacro\topcoord{#4};
        \pgfmathsetmacro\bottomcoord{10^(-#2)*#4};
        \pgfkeys{/pgf/fpu=false}
        \coordinate (#1-TR) at (axis cs:\rightcoord,\topcoord);
        \coordinate (#1-BR) at (axis cs:\rightcoord,\bottomcoord);
        \coordinate (#1-TL) at (axis cs:\leftcoord,\topcoord);
        \shadedraw[%
            bottom color = black!20,%
            middle color = black!5,%
            top color    = white,%
            draw         = black,%
            font         = \footnotesize%
        ]
        (#1-BR) -- (#1-TR) node[midway, right] {\(#2\)}
        -- (#1-TL) node[midway, above] {\(1\)} -- (#1-BR);
    }
}
\newcommand\drawslopetriangleup[4][STU]{
  \pgfplotsextra
  {
    \pgfkeys{/pgf/fpu=true}
    \pgfmathsetmacro\leftcoord{#3};
    \pgfmathsetmacro\rightcoord{10*#3};
    \pgfmathsetmacro\bottomcoord{#4};
    \pgfmathsetmacro\topcoord{10^(#2)*#4};
    \pgfkeys{/pgf/fpu=false}

    \coordinate (#1-BL) at (axis cs:\leftcoord,\bottomcoord);
    \coordinate (#1-BR) at (axis cs:\rightcoord,\bottomcoord);
    \coordinate (#1-TR) at (axis cs:\rightcoord,\topcoord);

    \shadedraw[%
      bottom color    = black!20,%
      middle color = black!5,%
      top color = white,%
      draw         = black%
    ]
      (#1-BL) -- (#1-BR) node[midway, below] {\(1\)}
      -- (#1-TR) node[midway, right] {\(#2\)} -- (#1-BL);
  }
}
\newcommand\drawswappedslopetriangleup[4][SSTU]{
  \pgfplotsextra
  {
    \pgfkeys{/pgf/fpu=true}
    \pgfmathsetmacro\leftcoord{#3};
    \pgfmathsetmacro\rightcoord{10*#3};
    \pgfmathsetmacro\bottomcoord{10^(-#2)*#4};
    \pgfmathsetmacro\topcoord{#4};
    \pgfkeys{/pgf/fpu=false}

    \coordinate (#1-BL) at (axis cs:\leftcoord,\bottomcoord);
    \coordinate (#1-TR) at (axis cs:\rightcoord,\topcoord);
    \coordinate (#1-TL) at (axis cs:\leftcoord,\topcoord);

    \shadedraw[%
      bottom color    = black!20,%
      middle color = black!5,%
      top color = white,%
      draw         = black%
    ]
      (#1-TR) -- (#1-TL) node[midway, above] {\(1\)}
      -- (#1-BL) node[midway, left] {\(#2\)} -- (#1-TR);
  }
}

\newcommand\annotateeast[4][1]{
    \pgfplotsextra
    {
        \pgfkeys{/pgf/fpu=true}
        \pgfmathsetmacro\tipx{#2};
        \pgfmathsetmacro\tipy{#3};
        \pgfkeys{/pgf/fpu=false}
        \draw[<-,gray] (axis cs: \tipx, \tipy) --
        +($#1*(10pt, 0pt)$) node[right,gray,font=\footnotesize] {#4};
    }
}
\newcommand\annotatesoutheast[4][1]{
    \pgfplotsextra
    {
        \pgfkeys{/pgf/fpu=true}
        \pgfmathsetmacro\tipx{#2};
        \pgfmathsetmacro\tipy{#3};
        \pgfkeys{/pgf/fpu=false}
        \draw[<-,gray] (axis cs: \tipx, \tipy) --
        +($#1*(8pt, -8pt)$) node[right,gray,font=\footnotesize] {#4};
    }
}
\newcommand\annotatenortheast[4][1]{
    \pgfplotsextra
    {
        \pgfkeys{/pgf/fpu=true}
        \pgfmathsetmacro\tipx{#2};
        \pgfmathsetmacro\tipy{#3};
        \pgfkeys{/pgf/fpu=false}
        \draw[<-,gray] (axis cs: \tipx, \tipy) --
        +($#1*(8pt, 8pt)$) node[right,gray,font=\footnotesize] {#4};
    }
}
\let\annotatesoutheastabs\annotatesoutheast
\pgfplotsset{
    colormap={parula}{
        rgb=(0.2081,0.1663,0.5292)
        rgb=(0.2116,0.1898,0.5777)
        rgb=(0.2123,0.2138,0.627)
        rgb=(0.2081,0.2386,0.6771)
        rgb=(0.1959,0.2645,0.7279)
        rgb=(0.1707,0.2919,0.7792)
        rgb=(0.1253,0.3242,0.8303)
        rgb=(0.0591,0.3598,0.8683)
        rgb=(0.0117,0.3875,0.882)
        rgb=(0.006,0.4086,0.8828)
        rgb=(0.0165,0.4266,0.8786)
        rgb=(0.0329,0.443,0.872)
        rgb=(0.0498,0.4586,0.8641)
        rgb=(0.0629,0.4737,0.8554)
        rgb=(0.0723,0.4887,0.8467)
        rgb=(0.0779,0.504,0.8384)
        rgb=(0.0793,0.52,0.8312)
        rgb=(0.0749,0.5375,0.8263)
        rgb=(0.0641,0.557,0.824)
        rgb=(0.0488,0.5772,0.8228)
        rgb=(0.0343,0.5966,0.8199)
        rgb=(0.0265,0.6137,0.8135)
        rgb=(0.0239,0.6287,0.8038)
        rgb=(0.0231,0.6418,0.7913)
        rgb=(0.0228,0.6535,0.7768)
        rgb=(0.0267,0.6642,0.7607)
        rgb=(0.0384,0.6743,0.7436)
        rgb=(0.059,0.6838,0.7254)
        rgb=(0.0843,0.6928,0.7062)
        rgb=(0.1133,0.7015,0.6859)
        rgb=(0.1453,0.7098,0.6646)
        rgb=(0.1801,0.7177,0.6424)
        rgb=(0.2178,0.725,0.6193)
        rgb=(0.2586,0.7317,0.5954)
        rgb=(0.3022,0.7376,0.5712)
        rgb=(0.3482,0.7424,0.5473)
        rgb=(0.3953,0.7459,0.5244)
        rgb=(0.442,0.7481,0.5033)
        rgb=(0.4871,0.7491,0.484)
        rgb=(0.53,0.7491,0.4661)
        rgb=(0.5709,0.7485,0.4494)
        rgb=(0.6099,0.7473,0.4337)
        rgb=(0.6473,0.7456,0.4188)
        rgb=(0.6834,0.7435,0.4044)
        rgb=(0.7184,0.7411,0.3905)
        rgb=(0.7525,0.7384,0.3768)
        rgb=(0.7858,0.7356,0.3633)
        rgb=(0.8185,0.7327,0.3498)
        rgb=(0.8507,0.7299,0.336)
        rgb=(0.8824,0.7274,0.3217)
        rgb=(0.9139,0.7258,0.3063)
        rgb=(0.945,0.7261,0.2886)
        rgb=(0.9739,0.7314,0.2666)
        rgb=(0.9938,0.7455,0.2403)
        rgb=(0.999,0.7653,0.2164)
        rgb=(0.9955,0.7861,0.1967)
        rgb=(0.988,0.8066,0.1794)
        rgb=(0.9789,0.8271,0.1633)
        rgb=(0.9697,0.8481,0.1475)
        rgb=(0.9626,0.8705,0.1309)
        rgb=(0.9589,0.8949,0.1132)
        rgb=(0.9598,0.9218,0.0948)
        rgb=(0.9661,0.9514,0.0755)
        rgb=(0.9763,0.9831,0.0538)
    }
}

\newcommand\vvvert{|\mkern-1.5mu|\mkern-1.5mu|}

\newcommand\D{\operatorname{D}\mkern-1mu}
\newcommand\E{\mathcal E}

\newcommand\N{\mathbb N}
\newcommand\R{\mathbb R}

\newcommand\T{\mathcal T}
\newcommand\V{\mathcal V}

\DeclareMathOperator\conv{conv}
\DeclareMathOperator\cond{cond}

\DeclareMathOperator\Int{int}
\DeclareMathOperator\Mid{mid}

\newcommand\ds{\, \mathrm{d} s}
\newcommand\dx{\, \mathrm{d} x}

\newcommand\IP{{\textup{IP}}}

\newcommand\PW{{\textup{pw}}}

\title{%
    Local parameter selection in the\\
    C\textsuperscript{0} interior penalty method
    for the biharmonic equation%
}

\author{%
    Philipp Bringmann%
    \thanks{Institute of Analysis and Scientific Computing,
    TU Wien, Austria \newline
    (\email{philipp.bringmann@asc.tuwien.ac.at}, \email{julian.streitberger@asc.tuwien.ac.at}).}
    \and
    Carsten Carstensen%
    \thanks{Department of Mathematics, Humboldt-Universit\"at zu Berlin,
        Germany
        (\email{cc@math.hu-berlin.de}).%
    }
    \and
    Julian Streitberger\footnotemark[1]%
}

\headers{Local parameter selection in C\textsuperscript{0}IP}{Bringmann, Carstensen, Streitberger}

\ifpdf
\hypersetup{%
    pdftitle={Local parameter selection in the
    C0 interior penalty method for the biharmonic equation}%
}
\fi
\begin{document}
    \maketitle
    %
    %
    \begin{abstract}
        The symmetric
        C\textsuperscript{0} interior penalty method
        is one of the most popular
        discontinuous Galerkin methods
        for the biharmonic equation.
        This paper introduces an automatic local selection
        of the involved stability parameter
        in terms of the geometry
        of the underlying triangulation
        for arbitrary polynomial degrees.
        The proposed choice ensures a stable discretization
        with guaranteed discrete ellipticity constant.
        Numerical evidence for uniform and
        adaptive mesh-refinement and various polynomial
        degrees supports the reliability and efficiency of the local parameter selection and recommends this in practice.
        The approach is documented in 2D for triangles,
        but the methodology behind can be generalized
        to higher dimensions,
        to non-uniform polynomial degrees,
        and to rectangular discretizations.
        Two appendices present the realization
        of our proposed parameter selection in various established
        finite element software packages as well as
        a detailed documentation of
        a self-contained MATLAB program for the lowest-order
        C\textsuperscript{0} interior penalty method.
    \end{abstract}

    \begin{keywords}
        C\textsuperscript{0} interior penalty method,
        discontinuous Galerkin method,
        biharmonic equation,
        implementation,
        local parameter selection, 
        penalty parameter
    \end{keywords}

    \begin{AMS}
        65N12, 65N15, 65N30, 65N50, 65Y20
    \end{AMS}
    %
    %

\section{Introduction}

A classical conforming finite element method for the plate problem
originates from the work of Argyris \cite{Argyris1968TheTF}
with a quintic polynomial ansatz space $P_5(T)$
and $21$ degrees of freedom on each triangle $T$.
The practical application appears less prominent
due to the higher computational efforts  \cite{DS2008}, although
recent works \cite{CH21, GRALE2022115352} establish
a highly efficient adaptive multilevel solver
for the hierarchical Argyris FEM.
%
At the moment, classical nonconforming FEMs
\cite{MR3061064,MR4230429,MR4235819,MR3407244}
and discontinuous Galerkin schemes
appear advantageous, at least more popular.

Discontinuous Galerkin finite element methods
\cite{MR431742,MR2142191,MR3061064,di2011mathematical,suli2007hp}
such as the C\textsuperscript{0} interior penalty method
are a popular choice for fourth-order problems
in order to avoid C\textsuperscript{1}-conforming finite elements.
Their main drawback is the dependence of the stability
on a sufficiently large penalty parameter.
The choice of this parameter is often heuristical and based on
the individual experience of the user.
This paper proposes an explicit geometry-dependent and local selection
of the penalty parameter that guarantees the stability of
the resulting scheme.

Given a source term $f \in L^2(\Omega)$ in a bounded polygonal
Lipschitz domain $\Omega \subset \R^2$,
let $u \in V \coloneqq H^2_0(\Omega)$ be the weak solution to
the biharmonic equation $\Delta^2 u = f$,
%
%
\begin{align}\label{eq:weakform}
    a(u,v)
    \coloneqq
    \int_\Omega \D^2 u : \D^2 v \dx
    =
    \int_\Omega f v \dx
    \quad \text{for all } v \in V.
\end{align}
The Riesz representation theorem
applies in the Hilbert space $(V,a)$ and proves
well-posedness of the formulation~\eqref{eq:weakform}.
Elliptic regularity theory verifies that
$f \in L^2(\Omega)$ implies $u \in H^{2+\alpha}(\Omega) \cap H^2_0(\Omega)$
\cite{MR2589244, MR595625, MR1814364, MR3014461}.
The pure Dirichlet boundary conditions
in the model example
lead to $\alpha > 1/2$.

For a regular triangulation \(\T\) of \(\Omega\)
into closed triangles and the polynomial degree \(k \geq 2\),
the symmetric C\textsuperscript{0} interior penalty method
(C0IP) seeks
\(
    u_\IP \in V_h
    \coloneqq
    S^k_0(\T)
    \coloneqq
    P_k(\T) \cap H^1_0(\Omega)
\)
in the Lagrange finite element space with
\begin{equation}
    \label{eq:discrete_form}
    A_h(u_\IP, v_\IP)
    =
    \int_\Omega f v_\IP \dx
    \quad \text{for all } v_\IP \in V_h.
\end{equation}
The three contributions to
the bilinear form $A_h: V_h \times V_h \to \R$ defined, for $u_\IP, v_\IP \in V_h$, by
\begin{equation}
    \label{eq:definition_Ah}
    A_h(u_{\IP}, v_{\IP})
    \coloneqq
    a_\PW(u_\IP,v_\IP)
    - (\mathcal{J}(u_\IP,v_\IP)
        + \mathcal{J}(v_\IP, u_\IP))
    + c_\IP(u_\IP, v_\IP)
\end{equation}
originate from the piecewise integration by parts
in the derivation of~\eqref{eq:weakform}
\cite{MR2142191}.
While the boundedness of \(A_h\)
follows from standard arguments,
its coercivity is subject to the assumption
of a sufficiently large penalty parameter $\sigma_{\IP,E} > 0$
in the bilinear form $c_\IP: V_h \times V_h \to \R$ defined, for \(u_\IP, v_\IP \in V_h\), by
\begin{equation}
    \label{eq:penalty_term}
    c_\IP(u_\IP,v_\IP)
    \coloneqq
    \sum_{E \in \mathcal{E}}
    \frac{\sigma_{\IP,E}}{h_E}
    \int_E
    [\nabla u_\IP \cdot \nu_E]_E
    [\nabla v _\IP \cdot \nu_E]_E
    \ds,
\end{equation}
with the normal jumps \([\nabla v_\IP \cdot \nu_E]_E\)
\cite{MR2142191,MR4023747}.
This paper introduces a local parameter
\begin{align}
	\label{eq:definition_sigma}
	\sigma_{\IP,E}
	\coloneqq
	\begin{cases}
		\displaystyle
		\frac{3a k(k-1) h_E^2}{8}
		\bigg(
		\frac{1}{\vert T_+\vert} + \frac{1}{\vert T_- \vert}
		\bigg)
		&
		\text{if }
		E = \partial T_+ \cap \partial T_- \in
		\E(\Omega),\medskip\\
		\displaystyle
		\frac{3a k(k-1) h_E^2}{2\vert T_+\vert}
		&
		\text{if }
		E \in \E(T_+) \cap \E(\partial\Omega).
	\end{cases}
\end{align}
The penalty parameter \(\sigma_{\IP,E}\) in~\eqref{eq:definition_sigma}
contains a prefactor $a > 1$.
Theorem~\ref{thm:stability} below
establishes that \emph{every} choice of \(a > 1\) leads
to guaranteed stability with stability constant at least \(\kappa = 1 - 1/\sqrt{a}\).
In the case of very large penalization with
$a \to \infty$, i.e.,
$\sigma_{\IP, E} \to \infty$,
the lower bound $\kappa$ tends to 1.
This fine-tuning of the penalty parameter \(\sigma_{\IP, E}\)
enables strong penalization as employed in \cite{MR2670003}
for the analysis of optimal convergence rates of
adaptive discontinuous Galerkin methods.

Numerical experiments
confirm the guaranteed stability
and exhibit rate-optimal convergence of
an adaptive C0IP for various polynomial degrees
in Section~\ref{sec:numerical_experiments} below.
The computation of the discrete $\inf$-$\sup$ constants (as certain eigenvalues of the discrete operator) reveals little overestimation only and recommends the proposed local parameter selection in practise.
A detailed investigation of the influence
of the parameter \(a\) reveals that
large penalty parameters lead to a substantial increase
of the condition numbers of the system matrix.
Hence, we recommend a small choice of \(a\), e.g.,
\(a = 2\), in order to avoid large condition numbers.
Alternatively, a strong penalization requires
the application of suitable preconditioners for C0IP
\cite{Brenner2005, Brenner2005a, Brenner2006, Brenner2011, Brenner2018, Brenner2022}.

\begin{example}
    On a uniform triangulation into right isosceles triangles of the same area,
    the penalty parameter~\eqref{eq:definition_sigma}
    for the quadratic C0IP method reads
    \[
        \sigma_{\IP, E}
        =
        \begin{cases}
            3a/4 & \text{for a cathetus } E \in \E(\Omega)
            \text{ in the interior,}
            \\
            3a/2 & \text{for a hypotenuse } E \in \E(\Omega)
            \text{ in the interior,}
            \\
            6a & \text{for a cathetus } E \in \E(\partial\Omega)
            \text{ on the boundary,}
            \\
            12a & \text{for a hypotenuse } E \in
            \E(\partial\Omega) \text{ on the boundary.}
        \end{cases}
    \]
    A choice of \(a\) close to \(1\) leads to
    smaller penalty parameters for the majority of interior
    edges compared to values from the literature,
    e.g., \(\sigma_{\IP, E} = 5\) in \cite{MR2670114}.
\end{example}

The remaining parts of the paper are organized as follows.
Section~\ref{sec:notation} specifies the notation such that
Section~\ref{sec:discrete_stability} can present the main
result of the paper in Theorem~\ref{thm:stability}.
Numerical experiments investigate the stability of the scheme with
the suggested automatic penalty selection
in Section~\ref{sec:numerical_experiments}.
The coercivity constants are computed numerically
and compared with the theoretically established value.
Moreover,
the performance of an adaptive mesh-refinement algorithm
is examined.
Appendix~\ref{sec:remarks_realization}
presents the numerical realization
of this automatic choice of the penalty parameter
to existing FEM software packages such as
the uniform form language \cite{OLW2009, Kirby2018, KM2018},
deal.II \cite{dealII94},
and NGSolve \cite{Schoeberl2014}.
The implementation of the C\textsuperscript{0}
interior penalty method in these packages
turns out to be very compact.
A complete and accessible documentation
of a self-contained MATLAB code
for the lowest-order discretization
supplements this
in Appendix~\ref{app:implementation}.

\section{Notation}
\label{sec:notation}
Standard notation of Lebesgue and Sobolev spaces, their norms,
and $L^2$ scalar products applies throughout the paper.
Instead of the space \(V\)
in the weak formulation~\eqref{eq:weakform},
discontinuous Galerkin methods employ
the piecewise Sobolev space
\begin{equation}
	\label{eq:piecewise_Sobolev}
	H^2(\T)
	\coloneqq
	\big\{
	v \in L^2(\Omega)
	:
	\forall T \in \T,\:
	v\vert_T
	\in H^2(T) \coloneqq H^2(\Int(T))
	\big\}
\end{equation}
with respect to a shape-regular triangulation \(\T\)
of the domain $\Omega$ into closed triangles.
Define the space of piecewise polynomials
\[
    P_k(\T)
    \coloneqq
    \{
        v \in L^2(\Omega)
        \;:\;
        v\vert_{T} \in P_k(T) \ \text{for all} \ T \in \T
    \}
\]
of total degree at most $k \in \N_0$.
For \(v \in H^2(\T)\),
the piecewise application of the distributional derivatives
leads to the piecewise Hessian \(\D^2_\PW v \in L^2(\Omega; \mathbb{S})\),
\((\D^2_\PW v)\vert_T \coloneqq \D^2 (v\vert_T)\)
with values in the space \(\mathbb{S} \subset \R^{2 \times 2}\)
of symmetric \(2 \times 2\) matrices.

The piecewise Sobolev functions \(v \in H^2(\T)\)
allow for the evaluation of
averages \(\langle v\rangle_E\) and
jumps \([v]_E\) across an edge \(E \in \E\).
Each interior edge \(E \in \E(\Omega)\)
of length \(h_E \coloneqq \vert E \vert\)
is the common edge of exactly two triangles \(T_+, T_- \in T\), written
\(E = \partial T_+ \cap \partial T_-\).
Then
\begin{equation}
	\label{eq:jump}
	[v]_E
	\coloneqq
	(v \vert_{T_+})\vert_E - (v \vert_{T_-})\vert_E
	\quad\text{and}\quad
	\langle v \rangle_E
	\coloneqq
	\frac12 (v \vert_{T_+})\vert_E + \frac12 (v \vert_{T_-})\vert_E.
\end{equation}
The unit normal vector \(\nu_E\) is oriented such that
\(\nu_E \cdot \nu_{T_\pm}\vert_E = \pm 1\)
for the outward unit normal vectors \(\nu_{T_\pm}\) of \(T_\pm\). For each boundary edge \(E \in \E(\partial\Omega)\),
let \(T_+ \in \T\) denote the unique triangle with edge
\(E \in \E(T_+)\) and set
\(
[v]_E
\coloneqq
\langle v \rangle_E
\coloneqq
(v \vert_{T_+})\vert_E.
\)
Analogous definitions apply for vector- or matrix-valued polynomials.
Let $\T$ denote a shape regular triangulation
of the polygonal Lipschitz domain $\Omega$ into closed triangles
and \(\V\) (resp.\ $\V(\Omega)$ or $\V(\partial \Omega)$)
the set of all (resp.\ interior or boundary) vertices \cite{MR2322235, MR2373954}.
Let \(\E\) (resp.\ $\E(\Omega)$ or $\E(\partial \Omega)$)
be the set of all (resp.\ interior or boundary) edges.
For each triangle \(T \in \T\) of area \(\vert T \vert\),
let \(\V(T)\) denote the set of its three vertices and
\(\E(T)\) the set of its three edges.
Abbreviate the edge patch
by \(\omega_E \coloneqq \Int(T_+ \cup T_-) \subseteq \Omega\)
for an interior edge
\(E = \partial T_+ \cap \partial T_- \in \E(\Omega)\)
and by
\(\omega_E \coloneqq \Int(T_+) \subseteq \Omega\)
for a boundary edge \(E \in \E(T_+) \cap \E(\partial\Omega)\).

The remaining contributions to the bilinear form \(A_h\) in~\eqref{eq:definition_Ah} consist of
the piecewise energy scalar product
\(a_\PW: V_h \times V_h \to \R\)
and the jump term \(\mathcal{J}: V_h \times V_h \to \R\)
defined, for \(u_\IP, v_\IP \in V_h\), by
\begin{equation}
	\label{eq:stiffness_and_jump}
	\begin{split}
		a_{\textup{pw}}(u_\IP,v_\IP)
		&\coloneqq
		\sum_{T \in \mathcal{T}}
		\int_T \D^2 u_\IP: \D^2 v_\IP \dx, \\
		\mathcal{J}(u_\IP,v_\IP)
		&\coloneqq
		\sum_{E \in \mathcal{E}}
		\int_E
		\langle (\D^2_\PW u_\IP \, \nu_E)  \cdot \nu_E \rangle_E
		\cdot [\nabla v_\IP \cdot \nu_E]_E \dx.
	\end{split}
\end{equation}
Note that a different
convention for the orientation
of the normal vector \(\nu_E\)
in \cite{MR2142191}
leads to other signs in the
bilinear form \(A_h\)
compared
to~\eqref{eq:definition_Ah}.
Given a basis \(\varphi_0, \dots, \varphi_J\) of
\(V_h\), define the matrices
\(B(\T), N(\T) \in \R^{J \times J}\), for \(j, k = 1,\dots, J\), by
\begin{equation}
	\label{eq:system_matrix}
	B_{jk}(\T)
	\coloneqq
	A_h(\varphi_j, \varphi_k)
	\quad\text{and}\quad
	N_{jk}(\T)
	\coloneqq
	a_\PW(\varphi_j, \varphi_k)
	+
	c_\IP(\varphi_j, \varphi_k).
\end{equation}

The notation $A \lesssim B$ abbreviates $A \le C B$ for a positive,
generic constant $C$, solely depending on the domain $\Omega$,
the polynomial degree, and the shape regularity of the triangulation $\T$;
$A \approx B$ abbreviates $A \lesssim B \lesssim A$.

For two indices \(j, k \in \N\),
let \(\delta_{jk} \in \{0, 1\}\) denote the Kronecker symbol
defined by \(\delta_{jk} \coloneqq 1\) if and only if \(j = k\).
The enclosing single bars \(\vert \bullet \vert\)
apply context-sensitively and
denote not only the modulus of real numbers,
the Euclidian norm of vectors in \(\R^2\),
but also the cardinality of finite sets,
the area of two-dimensional Lebesgue sets,
and the length of edges.

\section{Stability}
\label{sec:discrete_stability}
This section develops a \textit{novel} stability analysis
with a penalty parameter $\sigma_{\IP,E}$
in the bilinear form \(c_\IP\) from~\eqref{eq:penalty_term}.
Let the discrete space \(V_h\)
be equipped with the mesh-dependent norm
$\Vert \bullet \Vert_h$ defined,
for \(v_\IP \in V_h\), by
\[
    \Vert v_\IP \Vert_h^2
    \coloneqq
    \vvvert v_\IP \vvvert_{\PW}^2 + c_\IP(v_\IP, v_\IP).
\]
with the piecewise semi-norm
\(
    \vvvert \bullet \vvvert_{\PW}
    \coloneqq
    \vert \bullet \vert_{H^2(\T)}
    \coloneqq
    \Vert \D^2_\PW \bullet \Vert_{L^2(\Omega)}
\).
The parameter \(\sigma_{\IP,E}\) from \eqref{eq:definition_sigma} solely depends
on the underlying triangulation and
allows for a guaranteed stability of the bilinear from
\(A_h\) from~\eqref{eq:definition_Ah} with respect to the
mesh-dependent norm \(\Vert \bullet \Vert_h\)
in the following theorem.

\begin{theorem}
    \label{thm:stability}
    Given any $a > 1$,
    define the penalty parameter $\sigma_{\IP,E} > 0$ as in~\eqref{eq:definition_sigma}.
    For every choice of \(a > 1\),
    the constant \(\kappa \coloneqq 1 - 1/\sqrt{a} > 0\)
    and all discrete functions $v_\IP \in V_h$ satisfy
    the stability estimate
    \begin{equation}
        \label{eq:stability}
          \kappa\: \Vert v_\IP \Vert_h^2 \leq A_h(v_\IP, v_\IP).
    \end{equation}
\end{theorem}
Two remarks are in order
before the proof of the Theorem~\ref{thm:stability}
concludes this section.

\begin{remark}[variable polynomial degrees]
    \label{rem:variable-p}
    The assumption of uniform polynomial degree can be weakened,
    as follows for a different polynomial degree $k_T$
    on every triangle $T \in \T$.
    For an edge \(E \in \E\),
    an analogous argumentation leads
    to the choice of the edge-dependent parameter
    \[
        \sigma_{\IP,E}
        \coloneqq
        \begin{cases}
            \displaystyle
            \frac{3a h_E^2}{8}
            \bigg(
                \frac{k_{T_+}(k_{T_+}-1)}{\vert T_+\vert}
                + \frac{k_{T_-}(k_{T_-}-1)}{\vert T_- \vert}
            \bigg)
            &
            \text{if }
            E = \partial T_+ \cap \partial T_- \in
            \E(\Omega),\medskip\\
            \displaystyle
            \frac{3a k_{T_+}(k_{T_+}-1) h_E^2}{2\vert T_+\vert}
            &
            \text{if }
            E \in \E(T_+) \cap \E(\partial\Omega).
        \end{cases}
    \]
\end{remark}

\begin{remark}[generalization to rectangles]
    \label{rem:quads}
    Theorem~\ref{thm:stability} can be generalized
    to rectangular meshes.
    To this end, let \(\T\) denote a regular triangulation
    into closed rectangles and
    \(Q_k(\T)\) the space of continuous piecewise polynomials of
    partial degree at most \(k \in \N_0\).
    The sharp discrete trace inequality
    from~\cite[Eqn.~(C.20)]{HillewaertPhD}
    for rectangles \(T \in \T\) with edge \(E \in \E(T)\)
    reads,
    for all \(q_k \in Q_k(T)\),
    \[
        \Vert q_k \Vert_{L^2(E)}^2
        \leq
        \frac{(k+1)^2 h_E}{\vert T \vert}
        \Vert q_k \Vert_{L^2(T)}^2.
    \]
    Then, for edge \(E \in \E\) and \(a > 1\),
    an analog of Theorem~\ref{thm:stability}
    leads to the following choice
    of the edge-dependent parameter
    \begin{equation}
        \label{eq:sigma_rectangles}
        \sigma_{\IP, E}
        \coloneqq
        \begin{cases}
            \displaystyle
            a (k - 1)^2 h_E^2
            \bigg(
                \frac{1}{\vert T_+ \vert}
                +
                \frac{1}{\vert T_- \vert}
            \bigg)
            & \text{if }
            E = \partial T_+ \cap \partial T_-
            \in \E(\Omega),
            \\
            \displaystyle
            \frac{4 a (k - 1)^2 h_E^2}{\vert T_+ \vert}
            & \text{if }
            E \in \E(T_+) \cap \E(\partial\Omega).
        \end{cases}
    \end{equation}
\end{remark}

The proof of the Theorem~\ref{thm:stability} employs
a discrete trace inequality.

\begin{lemma}[{\cite[Thm.~3]{MR1986022}}]
	\label{lem:discrete_trace}
	Every polynomial \(q_k \in P_k(T)\)
	of degree at most \(k \in \N_0\) satisfies
	\[
	\Vert q_k \Vert_{L^2(E)}^2
	\leq
	\frac{(k+2)(k+1) h_E}{2 \vert T \vert}
	\Vert q_k \Vert_{L^2(T)}^2.
	\]
	The multiplicative constant is sharp
	in the sense that it cannot be replaced
	by any smaller constant in the absence of further
	conditions on the function \(q_k\).
\end{lemma}

\begin{proof}[Proof of Theorem~\ref{thm:stability}]
    The point of departure is the difference,
    for every $\kappa > 0$,
    \begin{equation}
        \label{eq:stability_kappa}
        \begin{split}
            A_h(v_\IP, v_\IP) - \kappa \Vert v_\IP \Vert_h^2
            &=
            (1-\kappa) \vvvert v_\IP \vvvert_\PW^2
            - 2 \mathcal{J}(v_\IP, v_\IP) \\
            &\hphantom{{}={}}+ (1 - \kappa) \sum_{E \in \mathcal{E}}
            \frac{\sigma_{\IP,E}}{h_E}
            \Vert [\nabla v_\IP \cdot \nu_E]_E \Vert_{L^2(E)}^2.
        \end{split}
    \end{equation}

    \emph{The first step} of the proof
    bounds the jump term $\mathcal{J}$
    with the weighted Young inequality,
    for every $\varepsilon > 0$,
    \[
        2 \vert \mathcal{J}(v_\IP, v_\IP) \vert
         \leq
        \sum \limits_{E \in \mathcal{E}}
        \Bigl(
            \frac{\varepsilon \sigma_{\IP,E}}{h_E}
            \Vert [\nabla v_\IP \cdot \nu_E]_E \Vert_{L^2(E)}^2
            + \frac{h_E}{\varepsilon \sigma_{\IP,E}}
            \Vert \langle
               (\D^2_\PW v_\IP \, \nu_E)  \cdot \nu_E
            \rangle_E \Vert_{L^2(E)}^2
        \Bigr).
    \]
    Inserting this estimate into~\eqref{eq:stability_kappa}
    confirms the lower bound
    \begin{equation}
        \label{eq:lower_bound_Ah}
        \begin{split}
            &
            (1-\kappa) \vvvert v_\IP \vvvert_\PW^2
            + (1-\kappa) \sum_{E \in \mathcal{E}}
            \frac{\sigma_{\IP,E}}{h_E}
            \Vert [\nabla v_\IP \cdot \nu_E]_E \Vert_{L^2(E)}^2\\
            &\hphantom{{}={}}
            - \sum_{E \in \mathcal{E}}
           \Bigl(
                \frac{\varepsilon \sigma_{\IP,E}}{h_E}
                \Vert [\nabla v_\IP \cdot \nu_E]_E \Vert_{L^2(E)}^2
                + \frac{h_E}{\varepsilon \sigma_{\IP,E}}
                \Vert \langle
                (\D^2_\PW v_\IP \, \nu_E)  \cdot \nu_E
                \rangle_E \Vert_{L^2(E)}^2
            \Bigr) \\
            &\hphantom{{}={}}\leq A_h(v_\IP, v_\IP) - \kappa \Vert v_\IP \Vert_h^2.
        \end{split}
    \end{equation}

    \emph{The second step}
    estimates the average $\langle (\D^2_\PW v_\IP \, \nu_E)  \cdot \nu_E \rangle_E$.
    Since the piecewise second derivative $\D^2_\PW v_{\IP}$
    of any function $v_\IP \in V_h$
    is a polynomial of degree \(k - 2\)
    on each triangle \(T\) with edge \(E \in \E(T)\),
    Lemma~\ref{lem:discrete_trace} implies that
    \begin{equation}
        \label{eq:Hessian_estimate}
        \Vert (\D^2_\PW v_\IP \, \nu_E)  \cdot \nu_E \vert_{T} \Vert_{L^2(E)}^2
        \leq
        \frac{k(k-1)h_E}{2\vert T \vert}
        \Vert (\D^2_\PW v_\IP \, \nu_E)  \cdot \nu_E \Vert_{L^2(T)}^2.
    \end{equation}
    In accordance with the bilinear form $a_\PW$ from \eqref{eq:stiffness_and_jump},
    the $L^2$ norm employs the Frobenius norm \(| \, \cdot \, |\) for matrix-valued function.
    A Cauchy-Schwarz inequality in $\R^2$, the submultiplicativity of the Frobenius norm,
    and the normalization $| \nu_E | = 1$ prove
    \[
        |
        (\D^2_\PW v_\IP(x) \, \nu_E)  \cdot \nu_E
        |^2
        \leq
        |
        \D^2 v_\IP(x)
        |^2
    \]
    for almost every $x \in E$.
     For an interior edge
    \(E=\partial T_+ \cap \partial T_- \in \E(\Omega)\)
    with the neighboring triangles \(T_+\) and \(T_-\), this and
    the definition of the average from~\eqref{eq:jump}
    \[
        \Vert \langle
           (\D^2_\PW v_\IP \, \nu_E)  \cdot \nu_E
        \rangle_E \Vert_{L^2(E)}^2
        \leq
        \frac{1}{4}
        \Vert
            \D^2 v_\IP \vert_{T_+} + \D^2 v_\IP \vert_{T_-}
        \Vert_{L^2(E)}^2.
    \]
    The combination with a triangle inequality,
    another weighted Young inequality with \(\alpha > 0\),
    and~\eqref{eq:Hessian_estimate}
    results in
    \begin{align*}
    	\Vert \langle (\D^2_\PW v_\IP \, \nu_E)  \cdot \nu_E&
    	\rangle_E \Vert_{L^2(E)}^2 \\
    	 &\leq
    	\frac{k(k-1)h_E}{8}
    	\Bigl(
    	\frac{1 + \alpha}{\vert T_+ \vert}
    	\Vert \D^2 v_\IP \Vert_{L^2(T_+)}^2
    	+
    	\frac{1 + 1/\alpha}{\vert T_- \vert}
    	\Vert \D^2 v_\IP \Vert_{L^2(T_-)}^2
    	\hspace{-3.1pt}
    	\Bigr).
    \end{align*}
    Hence
    the optimal value
    \(\alpha = \vert T_+ \vert / \vert T_- \vert\)
    leads to
    \[
        \Vert \langle (\D^2_\PW v_\IP \, \nu_E)  \cdot \nu_E
        \rangle_E \Vert_{L^2(E)}^2
        \leq
        \frac{k(k-1)h_E}{8}
        \bigg(
            \frac{1}{\vert T_+ \vert}
            + \frac{1}{\vert T_- \vert}
        \bigg)
        \Vert \D^2_\PW v_\IP \Vert_{L^2(\omega_E)}^2.
    \]
    The multiplication with
    \(h_E / (\varepsilon \sigma_{\IP, E})\)
    and the definition of \(\sigma_{\IP, E}\)
    in~\eqref{eq:definition_sigma} show
    \begin{equation}
        \label{eq:average_estimate}
        \frac{h_E}{\varepsilon \sigma_{\IP,E}}
        \Vert \langle (\D^2_\PW v_\IP \, \nu_E)  \cdot \nu_E
        \rangle_E \Vert_{L^2(E)}^2
        \leq
        \frac{1}{3 a \varepsilon}
        \Vert \D^2_\PW v_\IP \Vert_{L^2(\omega_E)}^2.
    \end{equation}

    For a boundary edge \(E \in \E(\partial \Omega)\) with
    adjacent triangle \(T_+ \in \T\),
    the definition of the average,
    the equality \(\vert \nu_E \vert = 1\),
    and~\eqref{eq:Hessian_estimate} verify
    \begin{align*}
        \Vert \langle
        (\D^2_\PW v_\IP \, \nu_E)  \cdot \nu_E
        \rangle_E \Vert_{L^2(E)}^2
        &=
        \Vert
            (\D^2_\PW v_\IP \, \nu_E)  \cdot \nu_E
        \Vert_{L^2(E)}^2
        \leq
        \Vert
            \D^2_\PW v_\IP
        \Vert_{L^2(E)}^2
        \\
        &\leq
        \frac{k(k-1)h_E}{2\vert T_+ \vert}
        \Vert \D^2 v_\IP \Vert_{L^2(T_+)}^2.
    \end{align*}
    Consequently,
    the multiplication with
    \(h_E / (\varepsilon \sigma_{\IP, E})\)
    and the definition of the penalty parameter \(\sigma_{\IP, E}\)
    result in the analogous estimate~\eqref{eq:average_estimate}
    for \(E \in \E(\partial\Omega)\).

    The sum of~\eqref{eq:average_estimate} over all edges \(E \in \E\)
    and the finite overlap of the edge patches
    $(\omega(E): E \in \mathcal{E})$ lead to
    \begin{equation}
        \label{eq:bound_hessian_normal}
        \sum_{E \in \mathcal{E}}
        \frac{h_E}{\varepsilon\sigma_{\IP,E}}
        \Vert \langle
            (\D^2_\PW v_\IP \, \nu_E)  \cdot \nu_E
        \rangle_E \Vert_{L^2(E)}^2
        \leq
        \frac{1}{a \varepsilon}
        \vvvert v_\IP \vvvert_\PW^2.
    \end{equation}

    \textit{The third step} concludes the proof with
    a combination of the lower bound~\eqref{eq:lower_bound_Ah}, and
    the estimate \eqref{eq:bound_hessian_normal} for
    \begin{equation}
        \label{eq:final_lower_bound}
        \begin{split}
            \left(1-\kappa - \frac{1}{a \varepsilon}\right)
            \vvvert v_\IP \vvvert_\PW^2
            + (1-\kappa - \varepsilon)
            \sum_{E \in \mathcal{E}}
            \frac{\sigma_{\IP,E}}{h_E}
            &\Vert [\nabla v_\IP \cdot \nu_E]_E \Vert_{L^2(E)}^2
            \\ &\hphantom{{}={}}\leq A_h(v_\IP, v_\IP) - \kappa \Vert v_\IP \Vert_h^2.
        \end{split}
    \end{equation}
    Every choice of
    \(0 < \kappa \leq \min\{1 - 1/(a \varepsilon), 1-\varepsilon\}\)
    leads to a nonnegative lower bound in~\eqref{eq:final_lower_bound}, and so proves
    the claim~\eqref{eq:stability};
     \(\varepsilon = 1/ \sqrt{a}\)
    results in \(\kappa = 1 - 1/\sqrt{a}\).
\end{proof}

\section{Numerical experiments}
\label{sec:numerical_experiments}

This section investigates an implementation of the C0IP
method from~\eqref{eq:discrete_form}
using an unpublished in-house software package
\texttt{AFEM2} in MATLAB version 9.9.0.1718557 (R2020b) Update 6.
The investigation considers
four computational benchmark examples with domains
displayed in Figure~\ref{fig:initialtriangulations}.
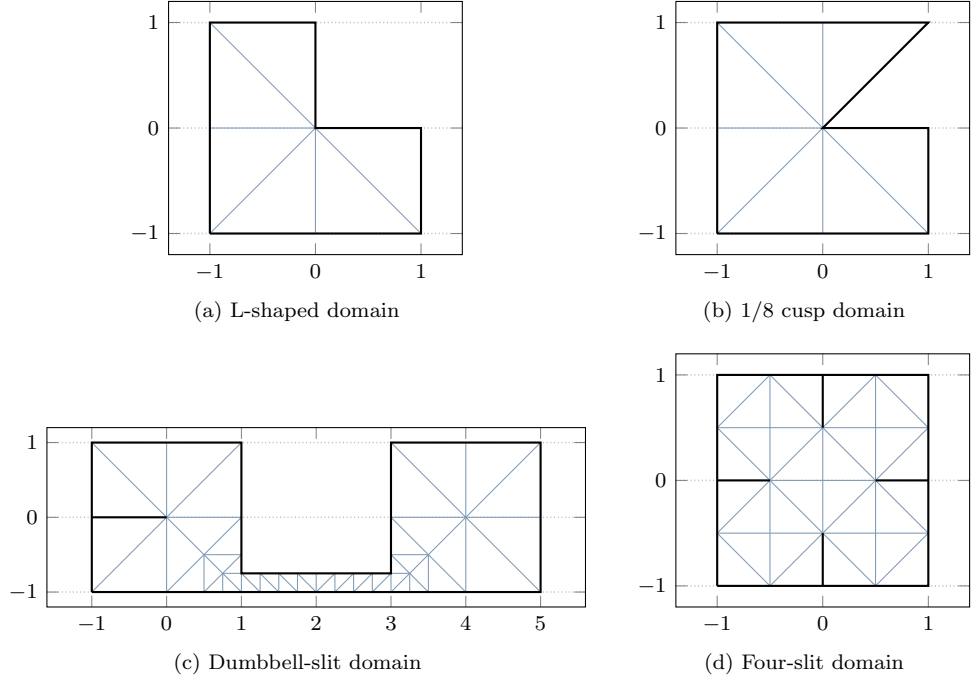
\begin{figure}
    \centering
    \hspace{1.5cm}
    \subfloat[L-shaped domain]{%
        \label{fig:Lshape}%
        \begin{tikzpicture}
		\colorlet{colEta}{HUred}
		\pgfplotsset{%
			eta/.style = {%
				colEta,%
				mark = *,%
				every mark/.append style={%
					solid,%
					scale = 1.1,%
					fill  = colEta!60!white
				}%
			}%
		}
		\pgfplotsset{%
			uniform/.style = {%
				dashed,
				every mark/.append style = {fill = black!20!white}
			},%
			adaptive/.style = {solid}%
		}

		\begin{axis}[%
			axis equal,%
			width            = 0.3\textwidth,%
			ymajorgrids      = true,%
			font             = \footnotesize,%
			legend style     = {
				legend columns = 1,
				legend pos     = outer north east,
				font           = \small
			}%
			]
		\addplot[HUblue!50!white] coordinates {(1,-1)(0,0)(0,-1)};

		\addplot[HUblue!50!white] coordinates {(-1,-1)(0,0)(-1,0)};

		\addplot[HUblue!50!white] coordinates {(-1,1)(0,0)};

		\addplot[black,thick] coordinates {(-1,-1)(-1,1)(0,1)(0,0)(1,0)(1,-1)(-1,-1)};
		\end{axis}
\end{tikzpicture}%
    }%
        \hfill
    \subfloat[1/8 cusp domain]{%
        \label{fig:cusp}%
        \begin{tikzpicture}
		\colorlet{colEta}{HUred}
		\pgfplotsset{%
			eta/.style = {%
				colEta,%
				mark = *,%
				every mark/.append style={%
					solid,%
					scale = 1.1,%
					fill  = colEta!60!white
				}%
			}%
		}
		\pgfplotsset{%
			uniform/.style = {%
				dashed,
				every mark/.append style = {fill = black!20!white}
			},%
			adaptive/.style = {solid}%
		}

		\begin{axis}[%
			axis equal,%
			width            = 0.3\textwidth,%
			ymajorgrids      = true,%
			font             = \footnotesize,%
			legend style     = {
				legend columns = 1,
				legend pos     = outer north east,
				font           = \small
			}%
			]
       			\addplot[HUblue!50!white] coordinates {(1,-1)(0,0)(0,-1)};

       			\addplot[HUblue!50!white] coordinates {(-1,-1)(0,0)(-1,0)};

                        \addplot[HUblue!50!white] coordinates {(-1,1)(0,0)(0,1)};

       			\addplot[black,thick] coordinates {(-1,-1)(-1,1)(1,1)(0,0)(1,0)(1,-1)(-1,-1)};
		\end{axis}
\end{tikzpicture}%
    }

    \subfloat[Dumbbell-slit domain]{%
        \label{fig:dumbbell}%
        \begin{tikzpicture}
		\colorlet{colEta}{HUred}
		\pgfplotsset{%
			eta/.style = {%
				colEta,%
				mark = *,%
				every mark/.append style={%
					solid,%
					scale = 1.1,%
					fill  = colEta!60!white
				}%
			}%
		}
		\pgfplotsset{%
			uniform/.style = {%
				dashed,
				every mark/.append style = {fill = black!20!white}
			},%
			adaptive/.style = {solid}%
		}

		\begin{axis}[%
			width            = 0.55\textwidth,%
			ymajorgrids      = true,%
			font             = \footnotesize,%
			axis equal image,%
			legend style     = {
				legend columns = 1,
				legend pos     = outer north east,
				font           = \small
			}%
			]
			\addplot[HUblue!50!white] coordinates {(-1,-1)(0,0)(0,-1)(1,0)(0,0)(1,1)};
			\addplot[HUblue!50!white] coordinates {(-1,1)(0,0)(0,1)};
			\addplot[HUblue!50!white] coordinates {(0.5,-1)(0.5,-0.5)(1,-0.5)(0.5,-1)};
			\addplot[HUblue!50!white] coordinates {(0.75,-1)(0.75,-0.75)(1,-0.75)};

			\addplot[HUblue!50!white] coordinates {(0,0)(1,-1)(1,-0.75)(1.25,-1)(1.25,-0.75)(1.5,-1)(1.5,-0.75)(1.75,-1)(1.75,-0.75)(2,-1)(2,-0.75)};
			\addplot[HUblue!50!white] coordinates {(2,-1)(2.25,-0.75)(2.25,-1)(2.5,-0.75)(2.5,-1)(2.75,-0.75)(2.75,-1)(3,-0.75)(3,-1)};

			\addplot[HUblue!50!white] coordinates {(5,-1)(4,0)(4,-1)(3,0)(4,0)(3,1)};
			\addplot[HUblue!50!white] coordinates {(3,-1)(5,1)};
			\addplot[HUblue!50!white] coordinates {(3,-0.75)(3.25,-0.75)(3.25,-1)};
			\addplot[HUblue!50!white] coordinates {(3.5,-1)(3.5,-0.5)(3,-0.5)(3.5,-1)};
			\addplot[HUblue!50!white] coordinates {(5,0)(4,0)(4,1)};

			\addplot[black,thick] coordinates {(-1,-1)(5,-1)(5,1)(3,1)(3,-0.75)(1,-0.75)(1,1)(-1,1)(-1,-1)};
			\addplot[black,thick] coordinates {(-1,0)(0,0)};
		\end{axis}
\end{tikzpicture}%
    }
        \hfill
    \subfloat[Four-slit domain]{%
        \label{fig:fourslit}%
        \begin{tikzpicture}
		\colorlet{colEta}{HUred}
		\pgfplotsset{%
			eta/.style = {%
				colEta,%
				mark = *,%
				every mark/.append style={%
					solid,%
					scale = 1.1,%
					fill  = colEta!60!white
				}%
			}%
		}
		\pgfplotsset{%
			uniform/.style = {%
				dashed,
				every mark/.append style = {fill = black!20!white}
			},%
			adaptive/.style = {solid}%
		}

		\begin{axis}[%
			axis equal,%
			width            = 0.3\textwidth,%
			ymajorgrids      = true,%
			font             = \footnotesize,%
			legend style     = {
				legend columns = 1,
				legend pos     = outer north east,
				font           = \small
			}%
			]
			\addplot[HUblue!50!white] coordinates {(-1,-0.5)(1,-0.5)};
			\addplot[HUblue!50!white] coordinates {(-1,0.5)(1,0.5)};

			\addplot[HUblue!50!white] coordinates {(-0.5,-1)(-0.5,1)};
			\addplot[HUblue!50!white] coordinates {(0.5,-1)(0.5,1)};

			\addplot[HUblue!50!white] coordinates {(-0.5,0)(0.5,0)};
			\addplot[HUblue!50!white] coordinates {(0,-0.5)(0,0.5)};

			\addplot[HUblue!50!white] coordinates {(-0.5,-1)(1,0.5)(0.5,1)(-1,-0.5)(-0.5,-1)};
			\addplot[HUblue!50!white] coordinates {(0.5,-1)(1,-0.5)(-0.5,1)(-1,0.5)(0.5,-1)};

			\addplot[black,thick] coordinates {(-1,-1)(1,-1)(1,1)(-1,1)(-1,-1)};

			\addplot[black,thick] coordinates {(0,-1)(0,-0.5)};
			\addplot[black,thick] coordinates {(0,0.5)(0,1)};
			\addplot[black,thick] coordinates {(-1,0)(-0.5,0)};
			\addplot[black,thick] coordinates {(0.5,0)(1,0)};
		\end{axis}
\end{tikzpicture}%
    }

    \caption{%
        Visualization of the initial triangulations $\mathcal{T}_0$
        for the four benchmark problems.
    }
    \label{fig:initialtriangulations}
\end{figure}

\subsection{Adaptive mesh-refinement}
\label{sec:AFEM}

The experiments include a comparison
of uniform and adaptive mesh-refinement
driven by the a~posteriori error estimator
\(\eta^2(\T) \coloneqq \sum_{T \in \T} \eta^2(T)\)
with, for all \(T \in \T\),
\begin{equation}
    \label{eq:estimator}
    \begin{split}
        \eta^2(T)
        &\coloneqq
        \Vert h_T^2 (f - \Delta^2 u_\IP) \Vert_{L^2(T)}^2
        +
        \sum_{E \in \E(T)}
        \frac{\sigma_{\IP,E}^2}{h_E}
        \Vert [ \nabla u_\IP ]_E \cdot \nu_E \Vert_{L^2(E)}^2\\
        &\hphantom{{}={}}
        + \!
        \sum_{E \in \E(T) \cap \mathcal{E}(\Omega)}
        h_E
        \Vert [(\D^2_\PW v_\IP \, \nu_E)  \cdot \nu_E]_E \Vert_{L^2(E)}^2\\
        &\hphantom{{}={}}
         + \!
        \sum_{E \in \E(T) \cap \mathcal{E}(\Omega)}
        h_E^3
        \Vert [\partial(\Delta v_\IP) / \partial{\nu_E}]_E \Vert_{L^2(E)}^2
    \end{split}
\end{equation}
from~\cite[Sect.~7.1]{MR3051409}.
The software employs the D\"orfler marking strategy~\cite{MR1393904} and
newest-vertex bisection according to~\cite{MR2353951}
in the adaptive Algorithm~\ref{alg}
to compute a sequence of nested triangulations
\((\T_\ell)_{\ell \in \N_0}\).
The marking strategy with minimal cardinality of
the set \(\mathcal{M}_\ell\) of marked triangles
can be realized in linear computational complexity
\cite{MR4136545}.
\begin{algorithm}
    \begin{algorithmic}
        \Require
        regular triangulation $\T_0$ and bulk parameter $0 < \theta \leq 1$.
        \For{$\ell = 0,1,2,\dots$}
        \State \textbf{Solve} \eqref{eq:discrete_form} with respect to triangulation
        $\T_\ell$ for solution $u_\ell \in S^k_0(\T_\ell)$.
        \State \textbf{Compute} refinement indicator $\eta^2(\T_\ell, T)$
        from~\eqref{eq:estimator}
        for all \(T \in \T_\ell\).
        \State \textbf{Mark} a minimal subset
        $\mathcal{M}_\ell \subseteq \T_\ell$
        by the D\"{o}rfler criterion
        \[
            \theta\: \eta^2(\T_\ell)
            \leq
            \sum_{T \in \mathcal{M}_\ell} \eta^2(\T_\ell, T).
        \]
        \State \textbf{Refine} $\T_\ell$ to $\T_{\ell + 1}$
        by newest-vertex bisection
        such that $\mathcal{M}_\ell\subseteq \T_\ell \setminus \T_{\ell+1}$.
        \EndFor
        \Ensure
        sequence of triangulations $\T_\ell$ with $u_\ell$
        and $\eta(\T_\ell)$ for $\ell \in \mathbb{N}_0$.
    \end{algorithmic}
    \caption{Adaptive C\textsuperscript{0} interior penalty method}
    \label{alg}
\end{algorithm}

Another main aspect of the numerical investigation
concerns the influence of the parameter \(a > 1\)
in~\eqref{eq:definition_sigma}
to the stability of the scheme.
To this end,
the principal eigenvalue $\lambda_{1, \ell}$
to the following discrete eigenvalue problem is computed:
Seek \(\Phi_\IP \in V_\ell \coloneqq S^k_0(\T_\ell)\)
with $\Phi_\IP \neq 0$ and \(\mu \in \R\) such that
\begin{equation}
    \label{eq:EVP}
    A_h(\Phi_\IP,v_\IP)
    =
    \mu (a_\PW(\Phi_\IP,v_\IP) + c_\IP(\Phi_\IP,v_\IP))
    \quad
    \text{for all }
    v_\IP \in V_\ell.
\end{equation}
The Rayleigh-Ritz principle shows that
\(\kappa = 1 - 1/\sqrt{a}\) from Theorem~\ref{thm:stability}
provides a lower bound for
the principal eigenvalue $\lambda_{1, \ell}$. The practical realization of~\eqref{eq:EVP} employs the matrices \(B(\T_\ell), N(\T_\ell) \in \R^{J \times J}\) from \eqref{eq:system_matrix}.

\subsection{Singular solution on L-shaped domain}\label{ex:singular_solution_lshape}

This benchmark considers
the L-shaped domain $\Omega = (-1,1)^2 \setminus [0,1)^2$ with
interior angle $\omega = 3\pi/2$ at the reentrant corner.
This determines the noncharacteristic solution
$\alpha \coloneqq 0.5444837$ to
$\sin^2(\alpha \omega) = \alpha^2 \sin^2(\omega)$ in
\begin{equation}
    \label{eq:angle_component}
    \begin{split}
        g_{\alpha,\omega}(\varphi)
        &=
        \bigg(
            \frac{\sin((\alpha-1)\omega)}{\alpha-1}
            - \frac{\sin((\alpha+1)\omega)}{\alpha+1}
        \bigg)
        \big(
            \cos((\alpha-1)\varphi) - \cos((\alpha+1)\varphi)
        \big) \\
        &-
        \bigg(
            \frac{\sin((\alpha-1)\varphi)}{\alpha-1}
            - \frac{\sin((\alpha+1)\varphi)}{\alpha-1}
        \bigg)
        \big(
            \cos((\alpha-1)\omega) - \cos((\alpha+1)\omega)
        \big).
    \end{split}
\end{equation}
The exact singular solution in polar coordinates
from \cite[p.~107]{MR1173209} reads
\begin{align}
    \label{eq:singularsolutionLshape}
    u(r,\varphi)
    =
    (1 - r^2 \cos^2\varphi)^2 (1 - r^2 \sin^2\varphi)^2
    r^{1+\alpha} g_{\alpha, \omega}\Big(\varphi-\frac{\pi}{2}\Big).
\end{align}
The right-hand side \(f \coloneqq \Delta^2 u\) is computed accordingly.
The reduced regularity of the exact solution $u$
leads to the empirical convergence rate $\alpha/2$
for uniform mesh-refinement displayed by dashed lines in
Figure~\ref{fig:lshape_convergence}.
The adaptive refinement strategy
results in local mesh-refining at the reentrant corner
in Figure~\ref{fig:lshape_mesh}.
For all polynomial degrees \(k = 2, \dots, 5\),
the adaptive algorithm
recovers the optimal convergence rates
with respect to the number of degrees of freedom
(ndof) as displayed by the solid lines in
Figure~\ref{fig:lshape_convergence}.
\begin{figure}
    \centering
    \begin{tikzpicture}
    \begin{axis}[%
        axis equal image,%
        width = 8cm,%
        xmin = -1.1,%
        xmax = 1.1,%
        ymin = -1.1,%
        ymax = 1.1%
    ]

        \addplot graphics [xmin=-1, xmax=1, ymin=-1, ymax=1]
        {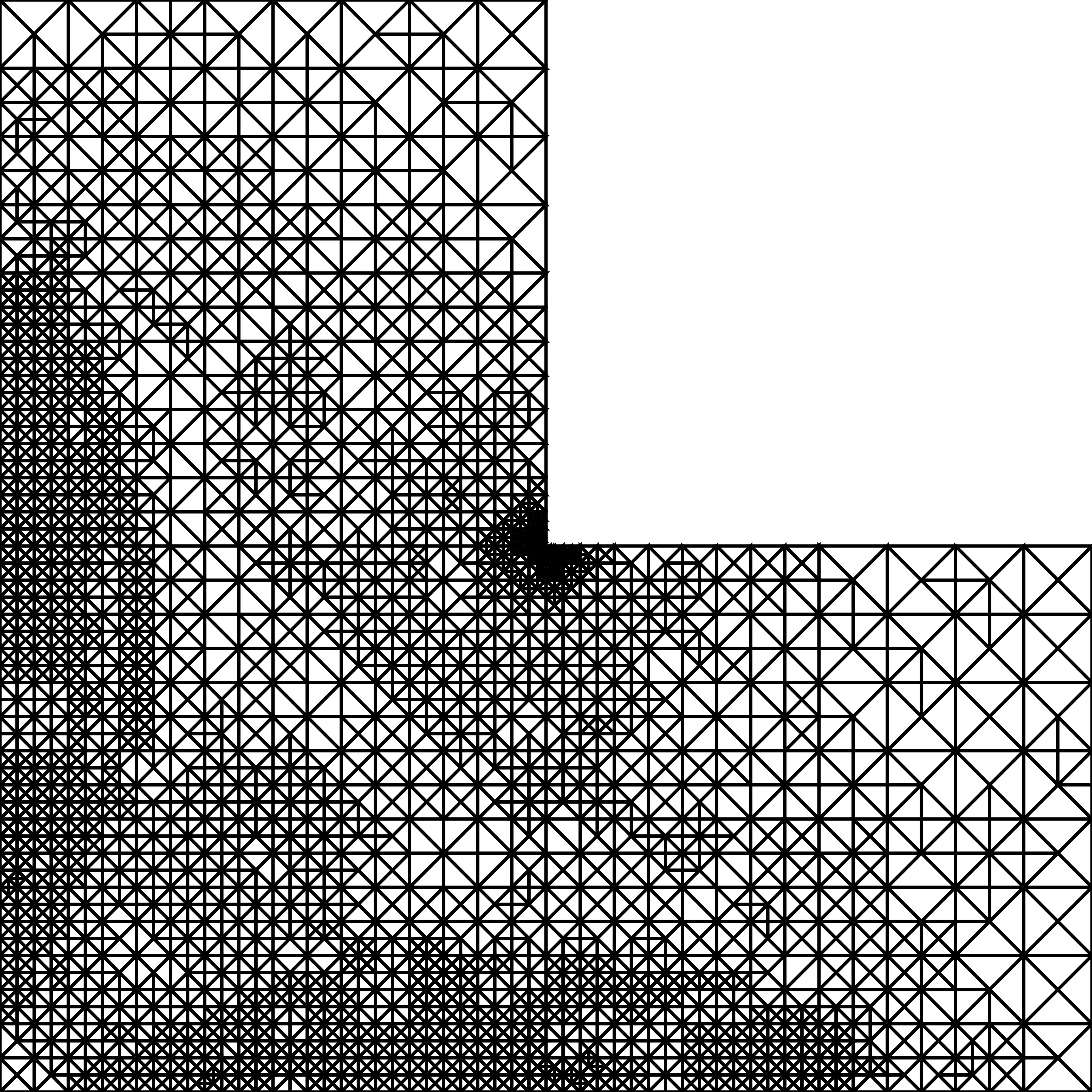};
    \end{axis}
\end{tikzpicture}%
    \caption{%
        Adaptively refined mesh (\(\theta = 0.5\))
        with polynomial degree \(k = 2\) and $10406$ degrees of freedom
        for the L-shaped domain
        in Section~\ref{ex:singular_solution_lshape}.
    }
    \label{fig:lshape_mesh}%
\end{figure}
\begin{figure}
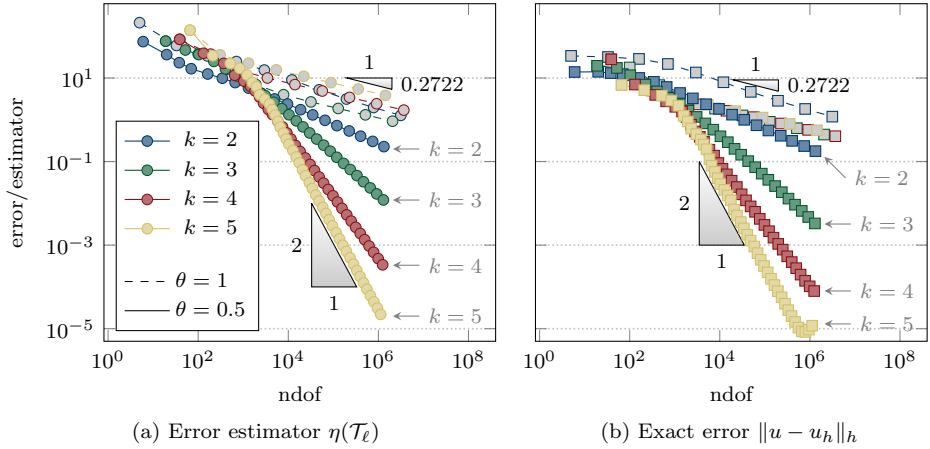

    \centering
    \subfloat[%
        Error estimator \(\eta(\T_\ell)\)
    ]{%
        \label{fig:lshape_eta}%
        \input{numerical_experiments/plot_lshape_convergence_eta}
    }
    \subfloat[%
        Exact error \(\Vert u - u_h \Vert_{h}\)
    ]{%
        \label{fig:lshape_error}%
        \input{numerical_experiments/plot_lshape_convergence_error}
    }
    \caption{%
        Convergence history plots
        for the L-shaped domain
        in Section~\ref{ex:singular_solution_lshape}.
    }
    \label{fig:lshape_convergence}
\end{figure}

\subsection{Singular solution on 1/8 cusp domain}
\label{ex:singular_solution_cusp}

This benchmark problem investigates the \(1/8\) cusp domain
$\Omega \coloneqq (-1,1)^2 \setminus \conv\{(0,0),(1,-1),(1,0)\}$
with interior angle $\omega = 7\pi / 4$
as depicted in Figure~\ref{fig:cusp}.
The right-hand side \(f \coloneqq \Delta^2 u\) is given by
the exact singular solution in polar coordinates
\begin{align}
    \label{eq:singular_solution_cusp}
    u(r,\varphi)
    =
    (1 - r^2 \cos^2\varphi)^2 (1 - r^2 \sin^2\varphi)^2
    r^{1+\alpha} g_{\alpha, \omega}\Big(\varphi - \frac\pi4\Big)
\end{align}
analogously to~\eqref{eq:singularsolutionLshape}
with \(g_{\alpha,\omega}\) from~\eqref{eq:angle_component}
for the parameter $\alpha = 0.50500969$.
The singularity causes an empirical convergence rate $\alpha/2$
for uniform mesh-refinement in Figure~\ref{fig:cusp_convergence},
while the adaptive algorithm recovers the optimal rates
for the considered polynomial degrees.
In the case \(k = 5\),
the highly adapted meshes lead to nearly singular matrices
close to the number of \(10^6\) degrees of freedom.
Undisplayed triangulation plots
confirm the increased adaptive refinement
towards the reentrant corner.
\begin{figure}
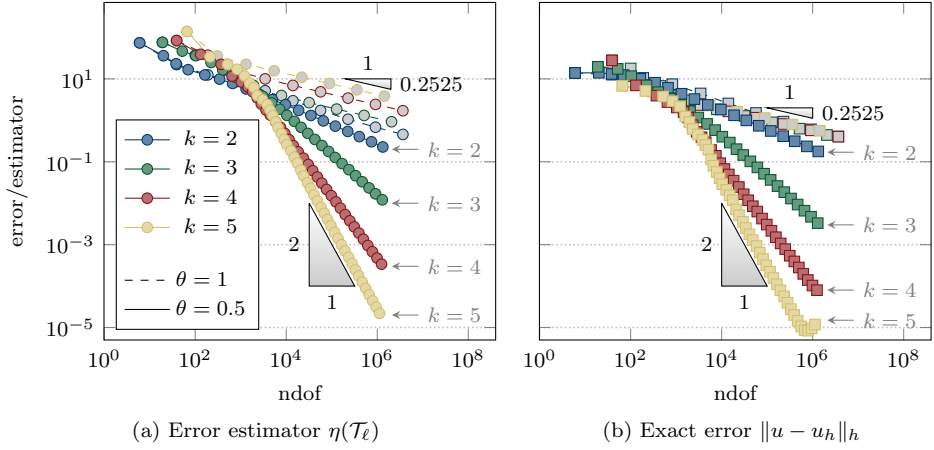

    \centering
    \subfloat[%
        Error estimator \(\eta(\T_\ell)\)
    ]{%
        \label{fig:cusp_eta}%
        \input{numerical_experiments/plot_cusp_convergence_eta}
    }
    \subfloat[%
        Exact error \(\Vert u - u_h \Vert_{h}\)
    ]{%
        \label{fig:cusp_error}%
        \input{numerical_experiments/plot_cusp_convergence_error}
    }
    \caption{%
        Convergence history plots
        for the \(1/8\) cusp domain
        in Section~\ref{ex:singular_solution_cusp}.
    }
    \label{fig:cusp_convergence}
\end{figure}

\subsection{Dumbbell-slit domain}
\label{ex:uniform_load_dumbbell}

The benchmark problem considers the uniform force
$f \equiv 1$ on the dumbbell-slit domain
\[
    \Omega
    =
    \Big(
        (-1,1) \times (-1,5)
        \setminus
        [1,3] \times [-0.75,1)
    \Big)
    \setminus
    (-1,0] \times \{0\}
\]
depicted in Figure~\ref{fig:dumbbell}.
Although \(\Omega\) is no Lipschitz domain,
Figures~\ref{fig:dumbbell_convergence}
confirm the optimal convergence rates with adaptive mesh-refinement.
Undisplayed plots of the adaptive triangulations
show an increased refinement towards the slit
as well as the two reentrant corners at
\((1, -0.75)\) and \((3, -0.75)\).

\begin{figure}
    \centering
    \label{fig:dumbbell_estimator}%
    \input{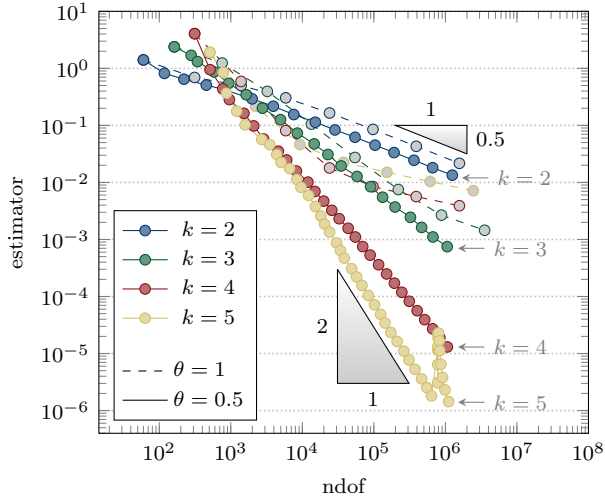}
    \caption{%
        Convergence history plots
        for the dumbbell domain
        in Section~\ref{ex:uniform_load_dumbbell}.
    }
    \label{fig:dumbbell_convergence}
\end{figure}

\subsection{Four-slit domain}
\label{ex:uniform_load_fourslit}
Let
\[
    \Omega
    =
    \Big(
        (-1,1)^2 \setminus
        \big(
            [-1,-0.5) \cup [0.5,1)
        \big)
        \times \{0\}
    \Big)
    \setminus \{0\} \times ([-1,-0.5) \cup [0.5,1))
\]
denote the square domain with one slit
at each edge as depicted in Figure~\ref{fig:fourslit}.
This benchmark also considers the
uniform load $f \equiv 1$.
Similarly to the previous benchmarks,
optimal convergence rates can be observed in the case
the adaptive mesh-refinement for the polynomial degrees
\(k = 2,\dots,5\) in Figure~\ref{fig:fourslit_convergence}.
%
%
Undisplayed triangulation plots indicate that the
adaptive algorithm detects the singularities
of the unknown exact solution
at tips of the four slits.

\begin{figure}
    \centering
    \label{fig:fourslit_estimator}%
    \input{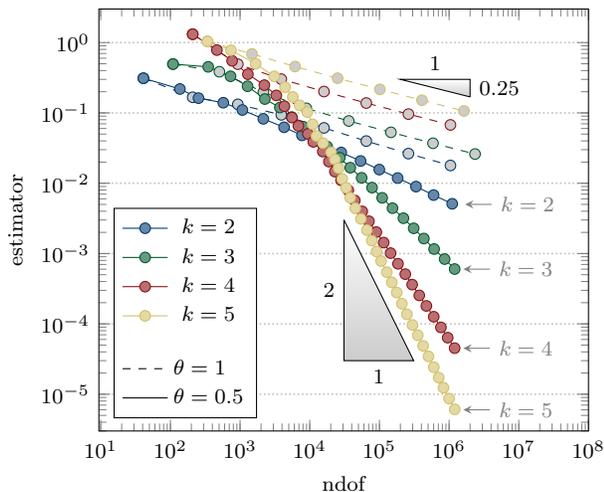}
    \caption{%
        Convergence history plots
        for the fourslit domain
        in Section~\ref{ex:uniform_load_fourslit}.
    }
    \label{fig:fourslit_convergence}
\end{figure}

\subsection{Investigation of parameter \textit{a}}

Given the triangulation \(\T_\ell\)
on the level \(\ell \in \N_0\)
from the adaptive Algorithm~\ref{alg},
the principle eigenvalue \(\lambda_{1, \ell}\)
of the discrete eigenvalue problem~\eqref{eq:EVP}
with respect to \(\T_\ell\)
is
the stability constant in~\eqref{eq:stability}.
The computation applies MATLAB's \texttt{eigs}
function to the eigenvalue problem of the matrices
\(B(\T_\ell)\) and \(N(\T_\ell)\)
from~\eqref{eq:system_matrix}.

Figure~\ref{fig:stability_plots_various_a}
displays $\lambda_{1, \ell}$
in dependence of the number of degrees of freedom
for different choices of \(a > 1\)
using adaptive and uniform mesh-refinement
on the L-shaped domain in Section~\ref{ex:singular_solution_lshape};
a large parameter \(a\) results
in a larger stability constant tending towards $1$.
The guaranteed lower bound \(\kappa = 1 - 1/\sqrt{a}\)
from Theorem~\ref{thm:stability} is always fulfilled.
It is remarkable that even the unjustified choice
\(a = 1\) leads to a seemingly stable discretization
with stability constants between \(0.2\) and \(0.5\) depicted in Figure~\ref{fig:stability_a10}.
Figure~\ref{fig:stability_variable_a} displays the
relation of the parameter \(a\) and the
stability constant \(\lambda_{1,\ell}\)
in more detail.
The computed values for \(\lambda_{1, \ell} < 1\) are slightly larger
than the guaranteed lower bounds.
This indicates little over-stabilization
of the automated choice of $\sigma_{\IP,E}$. Undisplayed numerical experiments do not reveal
any significant difference between the computed stability constants of the four domains in Figure~\ref{fig:initialtriangulations}.

An interesting empirical observation is that
adaptive mesh-refinement and higher-order polynomials
even lead to slightly more stable discretizations.
An over-sta\-bili\-zation as in the theoretical analysis
from \cite{MR2670003} is \emph{not} necessary in the computational benchmarks.
This supports the advantage of the automated choice
of the penalization with $\sigma_{\IP,E}$
for guaranteed stability.

The condition number \(\cond_1(B(\T_\ell))\)
with respect to the 1-norm
of the system matrix \(B(\T_\ell)\)
from~\eqref{eq:system_matrix} depends on \(a\)
as displayed in Figure~\ref{fig:condition_plots_various_k}.
The presented lower bounds for \(\cond_1(B(\T_\ell))\)
are computed using MATLAB's \texttt{condest} function.
Naturally, the number of degrees of freedom has the biggest
influence and the rate of the increase depends on the
polynomial degree \(k\).
While the importance of \(a\) seems neglectible,
the difference between the condition number for \(a = 1\)
and \(a = 100\) is still about two orders of magnitude.
Figure~\ref{fig:condition_variable_a} illustrates this
linear relation between the parameter \(a\) and the condition
number on a fixed uniform mesh of the L-shaped domain.
This suggests that a small \(a\)
is advantageous.

\begin{figure}
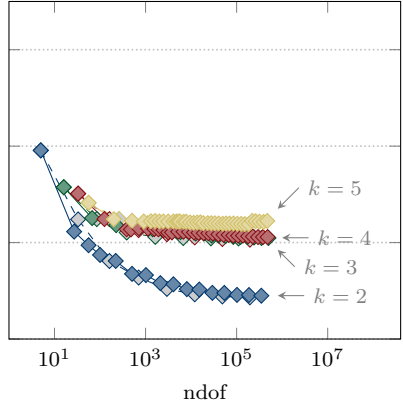
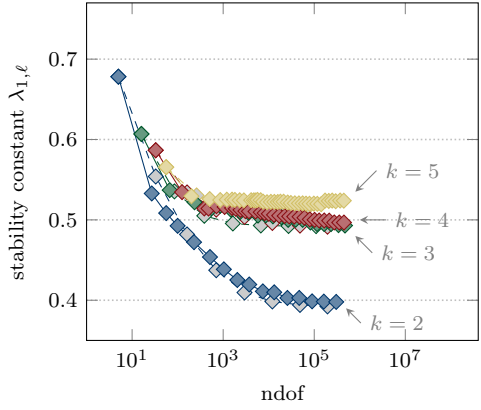
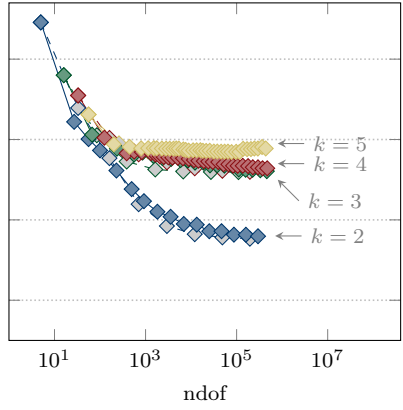
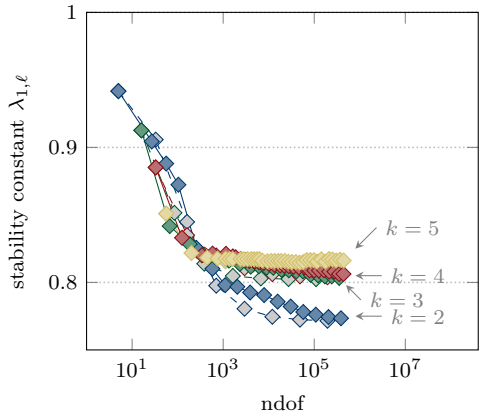
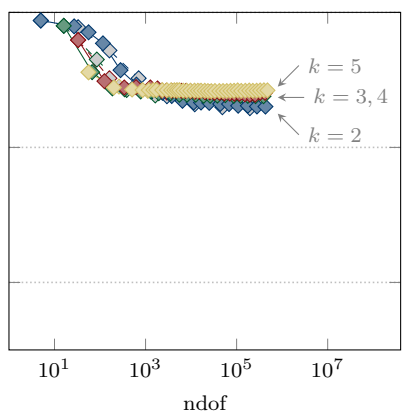

    \centering
    \subfloat[%
        \(a = 1.0\)
        (lower bound \(\kappa = 0\))
    ]{%
        \label{fig:stability_a10}%
        \input{numerical_experiments/plot_stability_a10}%
    }
    \hfill
    \subfloat[%
        \(a = 1.1\)
        (lower bound \(\kappa = 0.0465\))
    ]{%
        \label{fig:stability_a11}%
        \input{numerical_experiments/plot_stability_a11}%
    }

    \subfloat[%
        \(a = 1.5\)
        (lower bound \(\kappa = 0.1835\))
    ]{%
        \label{fig:stability_a15}%
        \input{numerical_experiments/plot_stability_a15}%
    }
    \hfill
    \subfloat[%
        \(a = 2.0\)
        (lower bound \(\kappa = 0.2929\))
    ]{%
        \label{fig:stability_a20}%
        \input{numerical_experiments/plot_stability_a20}%
    }

    \subfloat[%
        \(a = 10.0\)
        (lower bound \(\kappa = 0.6838\))
    ]{%
        \label{fig:stability_a100}%
        \input{numerical_experiments/plot_stability_a100}%
    }
    \hfill
    \subfloat[%
        \(a = 100.0\)
        (lower bound \(\kappa = 0.9\))
    ]{%
        \label{fig:stability_a1000}%
        \input{numerical_experiments/plot_stability_a1000}%
    }

    \caption{%
        Plot of stability constant $\lambda_{1, \ell}$ with
        various selections of parameter \(a\) for the L-shaped domain
        in Subsection~\ref{ex:singular_solution_lshape}.
    }
    \label{fig:stability_plots_various_a}
\end{figure}



\begin{figure}
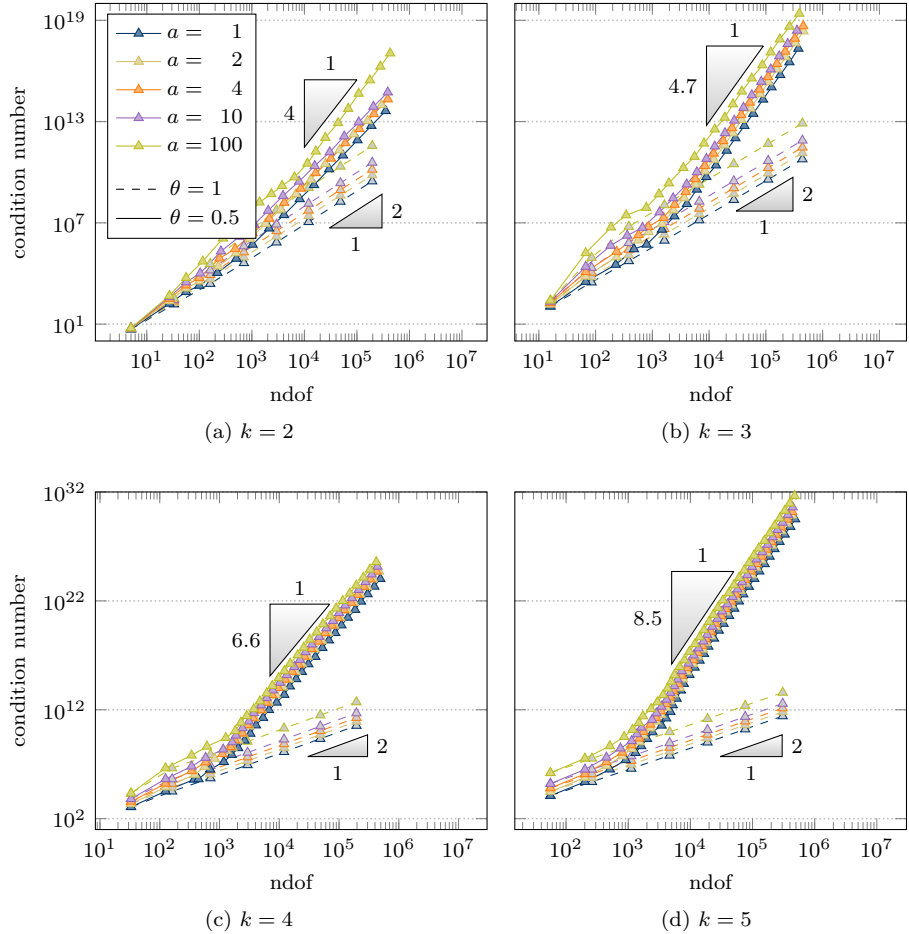

	\centering
	\subfloat[%
	\(k = 2\)
	]{%
		\label{fig:condition_k2}%
		\input{numerical_experiments/plot_condition_k2}%
	}
	\subfloat[%
	\(k = 3\)
	]{%
		\label{fig:condition_k3}%
		\input{numerical_experiments/plot_condition_k3}%
	}
	\hfill
	\subfloat[%
	\(k = 4\)
	]{%
		\label{fig:condition_k4}%
		\input{numerical_experiments/plot_condition_k4}%
	}
	\subfloat[%
	\(k = 5\)
	]{%
		\label{fig:condition_k5}%
		\input{numerical_experiments/plot_condition_k5}%
	}

	\caption{%
		Plot of condition number \(\cond_1(B(\T_\ell))\)
                of the system matrix $B(\T_\ell)$
                from~\eqref{eq:system_matrix}
                with various selections of \(a\) for the L-shaped domain
		in Subsection~\ref{ex:singular_solution_lshape}.
	}
	\label{fig:condition_plots_various_k}
\end{figure}

\begin{figure}
    \centering
    \subfloat[%
        Stability constant \(\lambda_{1, \ell} > 1-a^{-1/2}\)
    ]{
        \label{fig:stability_variable_a}%
        \begin{tikzpicture}[>=stealth]
    \colorlet{col2}{HUblue}
    \colorlet{col3}{HUgreen}
    \colorlet{col4}{HUred}
    \colorlet{col5}{HUsand}
    \pgfplotsset{%
        eta/.style = {%
            mark = *,%
            mark size = 2pt,%
            every mark/.append style = {solid}%
        },%
        err/.style = {%
            mark = square*,%
            mark size = 2pt,%
            every mark/.append style = {solid}%
        },%
        stab/.style = {%
            mark = square*,%
            mark size = 2pt,%
            every mark/.append style = {solid, rotate=45}%
        },%
        degdefault/.style = {%
            gray,%
            every mark/.append style = {fill = gray!60!white}%
        },%
        deg2/.style = {%
            col2,%
            every mark/.append style = {fill = col2!60!white}%
        },%
        deg3/.style = {%
            col3,%
            every mark/.append style = {fill = col3!60!white}%
        },%
        deg4/.style = {%
            col4,%
            every mark/.append style = {fill = col4!60!white}%
        },%
        deg5/.style = {%
            col5,%
            every mark/.append style = {fill = col5!60!white}%
        },%
        uniform/.style = {%
            dashed,%
            every mark/.append style = {%
                fill = black!20!white
            }%
        },%
        adaptive/.style = {%
            solid%
        }%
    }

    \begin{semilogxaxis}[%
            width            = 0.40\textwidth,%
            xlabel           = {parameter \(a\)},%
            ymajorgrids      = true,%
            font             = \footnotesize,%
            grid style       = {%
                densely dotted,%
                semithick%
            },%
            legend style     = {%
                legend pos  = south east %
            }%
        ]

        \pgfplotstableread[col sep=comma]{
a4lvl,ndof4lvl,lambda4lvl,condition4lvl
1,12033,0.256683330343277,11383661.3240317
2,12033,0.482224487495869,26048278.2088086
4,12033,0.638564875226253,54946207.2686323
8,12033,0.747149204316838,111555037.303415
16,12033,0.82294571351675,221370280.027916
32,12033,0.876149559412311,430779370.563176
64,12033,0.913721098274308,818505194.61071
128,12033,0.940435218646901,1504252763.80911
256,12033,0.959540494447431,2633700510.10662
512,12033,0.97323020719632,4298066815.75867
1024,12033,0.982986109520912,6395072960.64137
2048,12033,0.98975119737771,8554407355.43949
4096,12033,0.994171567424943,10359343630.2582
8192,12033,0.996844185142011,11617188669.077
16384,12033,0.998349124125518,12380017172.3161
32768,12033,0.999154278627042,12804022185.3887
65536,12033,0.999571812714167,13028171646.5251
131072,12033,0.999784518734896,13143500966.7354
        }\tableUniformTwo;

        \pgfplotstableread[col sep=comma]{
a4lvl,ndof4lvl,lambda4lvl,condition4lvl
1,27265,0.380601002477172,232174644.736644
2,27265,0.559643231483478,537958813.911282
4,27265,0.687931944019052,1151186713.56743
8,27265,0.77921068962665,2375381862.79106
16,27265,0.843899102177867,4820180272.9301
32,27265,0.889659299120625,9704728413.15316
64,27265,0.922006903717697,19467245652.8136
128,27265,0.944868718541644,38983481108.3211
256,27265,0.961026541453151,78002732806.1403
512,27265,0.972447126397566,156016161176.211
1024,27265,0.980520077097825,311987701437.09
2048,27265,0.986227113436393,623806876197.929
4096,27265,0.990261865044855,1247189677050.1
8192,27265,0.99311448789155,2493434974303.62
16384,27265,0.995131404548268,4984954217140.46
32768,27265,0.996557482414673,9966259942072.07
65536,27265,0.997565822327266,19927920670472.7
131072,27265,0.998278801202599,39844368970141.1
        }\tableUniformThree;

        \pgfplotstableread[col sep=comma]{
a4lvl,ndof4lvl,lambda4lvl,condition4lvl
1,48641,0.378388408317561,2175658794.38825
2,48641,0.561076959195849,5214988654.08821
4,48641,0.690190714971324,11545163361.7767
8,48641,0.781312815100468,24249660351.2607
16,48641,0.845601883629243,49648227222.9823
32,48641,0.890969301816574,100432358455.843
64,48641,0.922994269461294,201986267528.728
128,48641,0.945605564645889,405081228527.834
256,48641,0.96157084170243,811259366232.225
512,48641,0.972844431084127,1623604363018.54
1024,48641,0.980806912553541,3248292721603.47
2048,48641,0.986432555754729,6497603667000.68
4096,48641,0.990408180359446,12997446790542.5
8192,48641,0.99321873122668,25991277397567.6
16384,48641,0.995204471337998,52004497467073.4
32768,48641,0.996609477442258,104006565596965
65536,48641,0.997602276895685,208134340680346
131072,48641,0.998300162085307,415067305524918
        }\tableUniformFour;

        \pgfplotstableread[col sep=comma]{
a4lvl,ndof4lvl,lambda4lvl,condition4lvl
1,76161,0.409205659045816,16347377905.4842
2,76161,0.581360972267037,39981035703.8014
4,76161,0.703533166234963,87363980290.9273
8,76161,0.790171960112949,182334183163.035
16,76161,0.851552207092521,372264842060.065
32,76161,0.895003249320714,752068954556.853
64,76161,0.925746273385792,1511634224149.21
128,76161,0.947491480464784,3030712185836.22
256,76161,0.96286990236577,6068771943851.71
512,76161,0.973744804286713,12144758949587.6
1024,76161,0.981434730471565,24295266939226.2
2048,76161,0.986872380588311,48606099936704.9
4096,76161,0.990717386182546,97201310676625.2
8192,76161,0.993436211775551,194367969237607
16384,76161,0.995358708340711,388404041870770
32768,76161,0.996718115400272,777593014273833
65536,76161,0.99764655183297,1.54828861792012e+15
131072,76161,0.983904444013322,3.12372052618346e+15
        }\tableUniformFive;

        \addlegendimage{stab,deg2};
        \addlegendentry{\(k = 2\)};
        \addlegendimage{stab,deg3};
        \addlegendentry{\(k = 3\)};
        \addlegendimage{stab,deg4};
        \addlegendentry{\(k = 4\)};
        \addlegendimage{stab,deg5};
        \addlegendentry{\(k = 5\)};

        \addlegendimage{empty legend}
        \addlegendentry{\relax}

        \addlegendimage{thick}
        \addlegendentry{\(1 - 1/\sqrt{a}\)}

        \addplot+ [stab, deg2, adaptive, forget plot]
        table [x=a4lvl, y=lambda4lvl] {\tableUniformTwo};
        \addplot+ [stab, deg3, adaptive, forget plot]
        table [x=a4lvl, y=lambda4lvl] {\tableUniformThree};
        \addplot+ [stab, deg4, adaptive, forget plot]
        table [x=a4lvl, y=lambda4lvl] {\tableUniformFour};
        \addplot+ [stab, deg5, adaptive, forget plot]
        table [x=a4lvl, y=lambda4lvl] {\tableUniformFive};

        \addplot[domain=1:1e5, thick, samples=100] {1 - 1/sqrt(x)};


    \end{semilogxaxis}
\end{tikzpicture}%
    }
    \hfill
    \subfloat[%
        Condition number \(\cond_1(B(\T_\ell))\)
    ]{
        \begin{tikzpicture}[>=stealth]
    \colorlet{col2}{HUblue}
    \colorlet{col3}{HUgreen}
    \colorlet{col4}{HUred}
    \colorlet{col5}{HUsand}
    \pgfplotsset{%
        eta/.style = {%
            mark = *,%
            mark size = 2pt,%
            every mark/.append style = {solid}%
        },%
        err/.style = {%
            mark = square*,%
            mark size = 2pt,%
            every mark/.append style = {solid}%
        },%
        stab/.style = {%
            mark = square*,%
            mark size = 2pt,%
            every mark/.append style = {solid, rotate=45}%
        },%
        cond/.style = {%
            mark = triangle*,%
            mark size = 2pt,%
            every mark/.append style = {solid}%
        },%
        degdefault/.style = {%
            gray,%
            every mark/.append style = {fill = gray!60!white}%
        },%
        deg2/.style = {%
            col2,%
            every mark/.append style = {fill = col2!60!white}%
        },%
        deg3/.style = {%
            col3,%
            every mark/.append style = {fill = col3!60!white}%
        },%
        deg4/.style = {%
            col4,%
            every mark/.append style = {fill = col4!60!white}%
        },%
        deg5/.style = {%
            col5,%
            every mark/.append style = {fill = col5!60!white}%
        },%
        uniform/.style = {%
            dashed,%
            every mark/.append style = {%
                fill = black!20!white
            }%
        },%
        adaptive/.style = {%
            solid%
        }%
    }

    \begin{loglogaxis}[%
            width            = 0.40\textwidth,%
            xlabel           = {parameter \(a\)},%
            ymajorgrids      = true,%
            font             = \footnotesize,%
            grid style       = {%
                densely dotted,%
                semithick%
            },%
            legend style     = {%
                legend pos  = north west %
            }%
        ]

        \pgfplotstableread[col sep=comma]{
a4lvl,ndof4lvl,lambda4lvl,condition4lvl
1,12033,0.256683330343277,11383661.3240317
2,12033,0.482224487495869,26048278.2088086
4,12033,0.638564875226253,54946207.2686323
8,12033,0.747149204316838,111555037.303415
16,12033,0.82294571351675,221370280.027916
32,12033,0.876149559412311,430779370.563176
64,12033,0.913721098274308,818505194.61071
128,12033,0.940435218646901,1504252763.80911
256,12033,0.959540494447431,2633700510.10662
512,12033,0.97323020719632,4298066815.75867
1024,12033,0.982986109520912,6395072960.64137
2048,12033,0.98975119737771,8554407355.43949
4096,12033,0.994171567424943,10359343630.2582
8192,12033,0.996844185142011,11617188669.077
16384,12033,0.998349124125518,12380017172.3161
32768,12033,0.999154278627042,12804022185.3887
65536,12033,0.999571812714167,13028171646.5251
131072,12033,0.999784518734896,13143500966.7354
        }\tableUniformTwo;

        \pgfplotstableread[col sep=comma]{
a4lvl,ndof4lvl,lambda4lvl,condition4lvl
1,27265,0.380601002477172,232174644.736644
2,27265,0.559643231483478,537958813.911282
4,27265,0.687931944019052,1151186713.56743
8,27265,0.77921068962665,2375381862.79106
16,27265,0.843899102177867,4820180272.9301
32,27265,0.889659299120625,9704728413.15316
64,27265,0.922006903717697,19467245652.8136
128,27265,0.944868718541644,38983481108.3211
256,27265,0.961026541453151,78002732806.1403
512,27265,0.972447126397566,156016161176.211
1024,27265,0.980520077097825,311987701437.09
2048,27265,0.986227113436393,623806876197.929
4096,27265,0.990261865044855,1247189677050.1
8192,27265,0.99311448789155,2493434974303.62
16384,27265,0.995131404548268,4984954217140.46
32768,27265,0.996557482414673,9966259942072.07
65536,27265,0.997565822327266,19927920670472.7
131072,27265,0.998278801202599,39844368970141.1
        }\tableUniformThree;

        \pgfplotstableread[col sep=comma]{
a4lvl,ndof4lvl,lambda4lvl,condition4lvl
1,48641,0.378388408317561,2175658794.38825
2,48641,0.561076959195849,5214988654.08821
4,48641,0.690190714971324,11545163361.7767
8,48641,0.781312815100468,24249660351.2607
16,48641,0.845601883629243,49648227222.9823
32,48641,0.890969301816574,100432358455.843
64,48641,0.922994269461294,201986267528.728
128,48641,0.945605564645889,405081228527.834
256,48641,0.96157084170243,811259366232.225
512,48641,0.972844431084127,1623604363018.54
1024,48641,0.980806912553541,3248292721603.47
2048,48641,0.986432555754729,6497603667000.68
4096,48641,0.990408180359446,12997446790542.5
8192,48641,0.99321873122668,25991277397567.6
16384,48641,0.995204471337998,52004497467073.4
32768,48641,0.996609477442258,104006565596965
65536,48641,0.997602276895685,208134340680346
131072,48641,0.998300162085307,415067305524918
        }\tableUniformFour;

        \pgfplotstableread[col sep=comma]{
a4lvl,ndof4lvl,lambda4lvl,condition4lvl
1,76161,0.409205659045816,16347377905.4842
2,76161,0.581360972267037,39981035703.8014
4,76161,0.703533166234963,87363980290.9273
8,76161,0.790171960112949,182334183163.035
16,76161,0.851552207092521,372264842060.065
32,76161,0.895003249320714,752068954556.853
64,76161,0.925746273385792,1511634224149.21
128,76161,0.947491480464784,3030712185836.22
256,76161,0.96286990236577,6068771943851.71
512,76161,0.973744804286713,12144758949587.6
1024,76161,0.981434730471565,24295266939226.2
2048,76161,0.986872380588311,48606099936704.9
4096,76161,0.990717386182546,97201310676625.2
8192,76161,0.993436211775551,194367969237607
16384,76161,0.995358708340711,388404041870770
32768,76161,0.996718115400272,777593014273833
65536,76161,0.99764655183297,1.54828861792012e+15
131072,76161,0.983904444013322,3.12372052618346e+15
        }\tableUniformFive;

        \addlegendimage{cond,deg2};
        \addlegendentry{\(k = 2\)};
        \addlegendimage{cond,deg3};
        \addlegendentry{\(k = 3\)};
        \addlegendimage{cond,deg4};
        \addlegendentry{\(k = 4\)};
        \addlegendimage{cond,deg5};
        \addlegendentry{\(k = 5\)};


        \addplot+ [cond, deg2, adaptive, forget plot]
        table [x=a4lvl, y=condition4lvl] {\tableUniformTwo};
        \addplot+ [cond, deg3, adaptive, forget plot]
        table [x=a4lvl, y=condition4lvl] {\tableUniformThree};
        \addplot+ [cond, deg4, adaptive, forget plot]
        table [x=a4lvl, y=condition4lvl] {\tableUniformFour};
        \addplot+ [cond, deg5, adaptive, forget plot]
        table [x=a4lvl, y=condition4lvl] {\tableUniformFive};

        \drawslopetriangleup[ST1]{0.9}{1e1}{7e7}
        \drawswappedslopetriangleup[ST2]{1}{1e3}{5e14}

    \end{loglogaxis}
\end{tikzpicture}%
        \label{fig:condition_variable_a}%
    }

    \caption{
        Plots of \(a\)-dependence of quantities
        on a uniform mesh with
        6\,144 triangles of the L-shaped domain
        in Subsection~\ref{ex:singular_solution_lshape}.
    }
\end{figure}
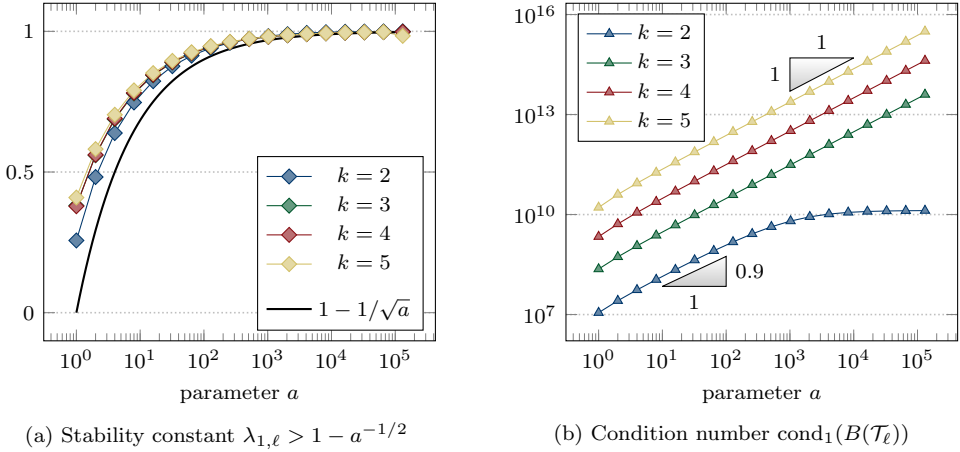

\subsection{Conclusions}

The numerical experiments validate
the theoretically established stability result
with the proposed penalty parameter from
\eqref{eq:definition_sigma}
for polynomial degrees $k=2,3,4,5$.
The Figures~\ref{fig:lshape_convergence}--%
\ref{fig:fourslit_convergence}
exhibit rate-optimal convergence of the
adaptive C\textsuperscript{0} interior penalty method.
%

A detailed investigation of the influence of the parameter
\(a > 1\) recommends the choice of \(a\) as small as possible
in order to avoid large condition numbers of the linear
system of equations.
We suggest \(a = 2\)
with \(\kappa = 1 - 1/\sqrt{2} \approx 0.293\).

\subsection*{Acknowledgements}
It is our pleasure to thank Zhaonan Dong (of Inria, Paris)
for fruitful discussions on the discrete trace inequality.
This paper has been supported by SPARC project (ID 235)
\emph{the mathematics and computation of plates}.



\bibliographystyle{siamplain}
\bibliography{references}

\begin{thebibliography}{10}

\bibitem{MR2589244}
{\sc S.~Agmon}, {\em Lectures on elliptic boundary value problems}, AMS Chelsea
  Publishing, Providence, RI, 2010.
\newblock Prepared for publication by B. Frank Jones, Jr. with the assistance
  of George W. Batten, Jr., Revised edition of the 1965 original.

\bibitem{MR3181899}
{\sc M.~S. Aln\ae~s, A.~Logg, K.~B. \O~lgaard, M.~E. Rognes, and G.~N. Wells},
  {\em Unified form language: a domain-specific language for weak formulations
  and partial differential equations}, ACM Trans. Math. Software, 40 (2014),
  pp.~Art. 9, 37.

\bibitem{fenics1.5}
{\sc M.~Aln\ae{}s, J.~Blechta, J.~Hake, A.~Johansson, B.~Kehlet, A.~Logg,
  C.~Richardson, J.~Ring, M.~E. Rognes, and G.~N. Wells}, {\em The {FEniCS}
  {Project} {Version} 1.5}, Archive of Numerical Software, 3 (2015), pp.~9--23.

\bibitem{Argyris1968TheTF}
{\sc J.~H. Argyris, I.~Fried, and D.~W. Scharpf}, {\em The tuba family of plate
  elements for the matrix displacement method}, Aeronautical Journal, 72
  (1968), pp.~701--709.

\bibitem{dealII94}
{\sc D.~Arndt, W.~Bangerth, M.~Feder, M.~Fehling, R.~Gassm{\"o}ller,
  T.~Heister, L.~Heltai, M.~Kronbichler, M.~Maier, P.~Munch, J.-P. Pelteret,
  S.~Sticko, B.~Turcksin, and D.~Wells}, {\em The \texttt{deal.II} library,
  version 9.4}, J. Numer. Math., 30 (2022), pp.~231--246.

\bibitem{MR431742}
{\sc G.~A. Baker}, {\em Finite element methods for elliptic equations using
  nonconforming elements}, Math. Comp., 31 (1977), pp.~45--59.

\bibitem{MR4189803}
{\sc P.~Bastian, M.~Blatt, A.~Dedner, N.-A. Dreier, C.~Engwer, R.~Fritz,
  C.~Gr\"{u}ninger, D.~Kempf, R.~Kl\"{o}fkorn, M.~Ohlberger, and O.~Sander},
  {\em The \textsc{Dune} framework: basic concepts and recent developments},
  Comput. Math. Appl., 81 (2021), pp.~75--112.

\bibitem{MR595625}
{\sc H.~Blum and R.~Rannacher}, {\em On the boundary value problem of the
  biharmonic operator on domains with angular corners}, Math. Methods Appl.
  Sci., 2 (1980), pp.~556--581.

\bibitem{MR2670003}
{\sc A.~Bonito and R.~H. Nochetto}, {\em Quasi-optimal convergence rate of an
  adaptive discontinuous {G}alerkin method}, SIAM J. Numer. Anal., 48 (2010),
  pp.~734--771.

\bibitem{MR2322235}
{\sc D.~Braess}, {\em Finite elements}, Cambridge University Press, Cambridge,
  third~ed., 2007.
\newblock Theory, fast solvers, and applications in elasticity theory.

\bibitem{MR3051409}
{\sc S.~C. Brenner}, {\em {$C^0$} interior penalty methods}, in Frontiers in
  numerical analysis---{D}urham 2010, vol.~85 of Lect. Notes Comput. Sci. Eng.,
  Springer, Heidelberg, 2012, pp.~79--147.

\bibitem{Brenner2011}
{\sc S.~C. Brenner, J.~Cui, T.~Gudi, and L.-Y. Sung}, {\em Multigrid algorithms
  for symmetric discontinuous {G}alerkin methods on graded meshes}, Numer.
  Math., 119 (2011), pp.~21--47.

\bibitem{Brenner2018}
{\sc S.~C. Brenner, C.~B. Davis, and L.-Y. Sung}, {\em Additive {S}chwarz
  preconditioners for the obstacle problem of clamped {K}irchhoff plates},
  Electron. Trans. Numer. Anal., 49 (2018), pp.~274--290.

\bibitem{MR2670114}
{\sc S.~C. Brenner, T.~Gudi, and L.-y. Sung}, {\em An a posteriori error
  estimator for a quadratic {$C^0$}-interior penalty method for the biharmonic
  problem}, IMA J. Numer. Anal., 30 (2010), pp.~777--798.

\bibitem{MR2373954}
{\sc S.~C. Brenner and L.~R. Scott}, {\em The mathematical theory of finite
  element methods}, vol.~15 of Texts in Applied Mathematics, Springer, New
  York, third~ed., 2008.

\bibitem{MR2142191}
{\sc S.~C. Brenner and L.-Y. Sung}, {\em {$C^0$} interior penalty methods for
  fourth order elliptic boundary value problems on polygonal domains}, J. Sci.
  Comput., 22/23 (2005), pp.~83--118.

\bibitem{Brenner2006}
{\sc S.~C. Brenner and L.-Y. Sung}, {\em Multigrid algorithms for {$C^0$}
  interior penalty methods}, SIAM J. Numer. Anal., 44 (2006), pp.~199--223.

\bibitem{Brenner2022}
{\sc S.~C. Brenner, L.-Y. Sung, and K.~Wang}, {\em Additive {S}chwarz
  preconditioners for {$C^0$} interior penalty methods for the obstacle problem
  of clamped {K}irchhoff plates}, Numer. Methods Partial Differential
  Equations, 38 (2022), pp.~102--117.

\bibitem{MR3061064}
{\sc S.~C. Brenner, L.-y. Sung, H.~Zhang, and Y.~Zhang}, {\em A {M}orley finite
  element method for the displacement obstacle problem of clamped {K}irchhoff
  plates}, J. Comput. Appl. Math., 254 (2013), pp.~31--42.

\bibitem{Brenner2005}
{\sc S.~C. Brenner and K.~Wang}, {\em Two-level additive {S}chwarz
  preconditioners for {$C^0$} interior penalty methods}, Numer. Math., 102
  (2005), pp.~231--255.

\bibitem{Brenner2005a}
{\sc S.~C. Brenner and J.~Zhao}, {\em Convergence of multigrid algorithms for
  interior penalty methods}, Appl. Numer. Anal. Comput. Math., 2 (2005),
  pp.~3--18.

\bibitem{C0IPMatlab}
{\sc P.~Bringmann, C.~Carstensen, and J.~Streitberger}, {\em Parameter-free
  implementation of the {C\textsuperscript{0}} interior penalty method for the
  biharmonic equation}, 2023.
\newblock Software to be published on \texttt{edoc.hu-berlin.de}.

\bibitem{CH21}
{\sc C.~Carstensen and J.~Hu}, {\em Hierarchical {A}rgyris finite element
  method for adaptive and multigrid algorithms}, Comput. Methods Appl. Math.,
  21 (2021), pp.~529--556.

\bibitem{MR4023747}
{\sc C.~Carstensen, G.~Mallik, and N.~Nataraj}, {\em {\it {A} priori} and {\it
  a posteriori} error control of discontinuous {G}alerkin finite element
  methods for the von {K}\'{a}rm\'{a}n equations}, IMA J. Numer. Anal., 39
  (2019), pp.~167--200.

\bibitem{MR4230429}
{\sc C.~Carstensen and N.~Nataraj}, {\em Adaptive {M}orley {FEM} for the von
  {K}\'{a}rm\'{a}n equations with optimal convergence rates}, SIAM J. Numer.
  Anal., 59 (2021), pp.~696--719.

\bibitem{MR4235819}
{\sc C.~Carstensen and N.~Nataraj}, {\em A priori and a posteriori error
  analysis of the {C}rouzeix-{R}aviart and {M}orley {FEM} with original and
  modified right-hand sides}, Comput. Methods Appl. Math., 21 (2021),
  pp.~289--315.

\bibitem{di2011mathematical}
{\sc D.~A. Di~Pietro and A.~Ern}, {\em Mathematical aspects of discontinuous
  Galerkin methods}, vol.~69, Springer, 2011.

\bibitem{DS2008}
{\sc V.~Dominguez and F.-J. Sayas}, {\em Algorithm 884: A simple matlab
  implementation of the \uppercase{A}rgyris element}, ACM Trans. Math. Softw.,
  35 (2008).

\bibitem{MR1393904}
{\sc W.~D{\"o}rfler}, {\em A convergent adaptive algorithm for {P}oisson's
  equation}, SIAM J. Numer. Anal., 33 (1996), pp.~1106--1124.

\bibitem{MR3407244}
{\sc D.~Gallistl}, {\em Morley finite element method for the eigenvalues of the
  biharmonic operator}, IMA J. Numer. Anal., 35 (2015), pp.~1779--1811.

\bibitem{MR1814364}
{\sc D.~Gilbarg and N.~S. Trudinger}, {\em Elliptic partial differential
  equations of second order}, Classics in Mathematics, Springer-Verlag, Berlin,
  2001.
\newblock Reprint of the 1998 edition.

\bibitem{MR1173209}
{\sc P.~Grisvard}, {\em Singularities in boundary value problems}, vol.~22 of
  Recherches en Math\'{e}matiques Appliqu\'{e}es [Research in Applied
  Mathematics], Masson, Paris; Springer-Verlag, Berlin, 1992.

\bibitem{GRALE2022115352}
{\sc B.~Gräßle}, {\em Optimal multilevel adaptive fem for the
  \uppercase{A}rgyris element}, Comput. Methods Appl. Mech. Eng., 399 (2022).

\bibitem{HillewaertPhD}
{\sc K.~Hillewaert}, {\em Development of the discontinuous {Galerkin} method
  for high-resultion, large scale {CFD} and acoustics in industrial
  geometries}, PhD thesis, Universit\'e catholique de Louvain, 2013.

\bibitem{Kirby2018}
{\sc R.~C. Kirby}, {\em A general approach to transforming finite elements},
  The SMAI Journal of computational mathematics, 4 (2018), pp.~197--224.

\bibitem{KM2018}
{\sc R.~C. Kirby and L.~Mitchell}, {\em Code generation for generally mapped
  finite elements}, ACM Trans. Math. Softw., 45 (2019).

\bibitem{MR3014461}
{\sc J.~Ne\v{c}as}, {\em Direct methods in the theory of elliptic equations},
  Springer Monographs in Mathematics, Springer, Heidelberg, 2012.

\bibitem{MR2466139}
{\sc K.~B. \O~lgaard, A.~Logg, and G.~N. Wells}, {\em Automated code generation
  for discontinuous {G}alerkin methods}, SIAM J. Sci. Comput., 31 (2008/09),
  pp.~849--864.

\bibitem{OLW2009}
{\sc K.~B. \O{}lgaard, A.~Logg, and G.~N. Wells}, {\em Automated code
  generation for discontinuous galerkin methods}, SIAM Journal on Scientific
  Computing, 31 (2009), pp.~849--864.

\bibitem{MR4136545}
{\sc C.-M. Pfeiler and D.~Praetorius}, {\em D\"{o}rfler marking with minimal
  cardinality is a linear complexity problem}, Math. Comp., 89 (2020),
  pp.~2735--2752.

\bibitem{MR3615280}
{\sc F.~Rathgeber, D.~A. Ham, L.~Mitchell, M.~Lange, F.~Luporini, A.~T.~T.
  McRae, G.-T. Bercea, G.~R. Markall, and P.~H.~J. Kelly}, {\em Firedrake:
  automating the finite element method by composing abstractions}, ACM Trans.
  Math. Software, 43 (2017).

\bibitem{Schoeberl2014}
{\sc J.~Sch\"oberl}, {\em C++ 11 implementation of finite elements in
  {NGSolve}}, Institute for analysis and scientific computing, TU Wien, 30
  (2014).

\bibitem{MR2353951}
{\sc R.~Stevenson}, {\em The completion of locally refined simplicial
  partitions created by bisection}, Math. Comp., 77 (2008), pp.~227--241.

\bibitem{suli2007hp}
{\sc E.~S{\"u}li and I.~Mozolevski}, {\em hp-version interior penalty
  \textsc{DGFEM}s for the biharmonic equation}, Comput. math. with appl., 196
  (2007), pp.~1851--1863.

\bibitem{MR1986022}
{\sc T.~Warburton and J.~S. Hesthaven}, {\em On the constants in {$hp$}-finite
  element trace inverse inequalities}, Comput. Methods Appl. Mech. Engrg., 192
  (2003), pp.~2765--2773.

\end{thebibliography}

\newpage
\appendix

\section{Remarks on the numerical realization}
\label{sec:remarks_realization}
%
The following three subsections illustrate
the numerical realization of the C0IP method
with the proposed local parameter selection
in established software packages
such as FEniCS (using the unified form language),
deal.II, and NGSolve.

\subsection{Realization using the unified form language (UFL)}
\label{sec:UFL}
The unified form language \cite{MR3181899}
provides a standardized way of implementing
variational formulations for PDEs.
It supports discontinuous Galerkin discretizations
such as the C0IP method \cite{MR2466139}.

The first step defines the shape of the element domain
and several geometric quantities required for the bilinear
form \(A_h\).
\begin{lstlisting}[language=python,numbers=none]
cell = triangle
n = FacetNormal(cell)
h = FacetArea(cell)
vol = CellVolume(cell)
\end{lstlisting}
Given a polynomial degree
\textup{k} \(\in \N\) with \textup{k} \(\geq 2\),
these variables allow to define the penalty parameter \(\sigma_{\IP,E}\)
according to the formula~\eqref{eq:definition_sigma}.
\begin{lstlisting}[language=python,numbers=none]
a = 4.0
sigma = 3.0 * a * k * (k - 1) / 8.0 * h("+")**2 * avg(1 / vol)
sigma_boundary = 3.0 * a * k * (k - 1) * h**2 / vol
\end{lstlisting}
The following lines specify the Lagrange finite element
and the trial and test functions.
\begin{lstlisting}[language=python,numbers=none]
element = FiniteElement("Lagrange", cell, k)
u = TrialFunction(element)
v = TestFunction(element)
\end{lstlisting}
Abbreviate the Hessian matrix of
finite element functions via
\begin{lstlisting}[language=python,numbers=none]
def hess(v):
    return grad(grad(v))
\end{lstlisting}
The realization of the bilinear form \(A_h\)
from~\eqref{eq:definition_Ah} reads
\begin{lstlisting}[language=python,numbers=none]
A = inner(hess(u), hess(v)) * dx \
    - inner(avg(dot(hess(u) * n, n)), jump(grad(v), n)) * dS \
    - inner(dot(grad(u), n), dot(hess(v) * n, n)) * ds \
    - inner(jump(grad(u), n), avg(dot(hess(v) * n, n))) * dS \
    - inner(dot(hess(u) * n, n), dot(grad(v), n)) * ds \
    + sigma / h("+") * inner(jump(grad(u), n), jump(grad(v), n)) * dS \
    + sigma_boundary / h * inner(dot(grad(u), n), dot(grad(v), n)) * ds
\end{lstlisting}
Finally, the right-hand side may be defined as follows.
\begin{lstlisting}[language=python,numbers=none]
f = Coefficient(element)
F = f * v * dx
\end{lstlisting}

Many software packages employ the UFL
such as DUNE \cite{MR4189803},
FEniCS \cite{fenics1.5},
and Firedrake \cite{MR3615280}.
With slight adaptions of the above UFL code,
the C0IP method with automatic penalty selection
can be realized within these three software frameworks
for arbitrary polynomial degree.

\subsection{Realization in deal.II}
The C++ software package
deal.II \cite{dealII94} (version 9.4)
provides a finite element implementation
for rectangular grids.
An example implementation of the C0IP method
in the file \texttt{examples/step-47.cc}
employs the following heuristic choice
of the penalty parameter
(adapted to the notation in this paper)
for an edge \(E \in \E\)
with \(d_\pm \coloneqq 2\vert T_\pm \vert / h_E\)
\[
    \frac{\widetilde\sigma_{\IP, E}}{h_E}
    =
    \begin{cases}
        \displaystyle
        k (k + 1)
        /
        \min\{d_+, d_-\}
        & \text{if }
        E = \partial T_+ \cap \partial T_-
        \in \E(\Omega),
        \\
        \displaystyle
        k (k + 1) / d_+
        & \text{if }
        E \in \E(T_+) \cap \E(\partial\Omega).
    \end{cases}
\]
This leads to a slight over-penalization
compared to the parameter suggested
in~\eqref{eq:sigma_rectangles}
\[
    \frac{\sigma_{\IP, E}}{h_E}
    =
    \begin{cases}
        \displaystyle
        a (k - 1)^2
        \bigg(
            \frac{1}{d_+}
            +
            \frac{1}{d_-}
        \bigg)
        & \text{if }
        E = \partial T_+ \cap \partial T_-
        \in \E(\Omega),
        \\
        \displaystyle
        \frac{4 a (k - 1)^2}{d_+}
        & \text{if }
        E \in \E(T_+) \cap \E(\partial\Omega).
    \end{cases}
\]
The application of the automatic parameter selection
with guaranteed stability solely requires
the replacement of the lines~512--518
in the file \texttt{step-47.cc} by
\begin{lstlisting}[numbers=none,language=c++,firstnumber=512]
const double a = 4.0;
const double sigma_over_h =
  2 * a * pow(k - 1, 2.0) * (
      1 / cell->extent_in_direction(
          GeometryInfo<dim>::unit_normal_direction[f]) +
      1 / ncell->extent_in_direction(
          GeometryInfo<dim>::unit_normal_direction[nf]));
\end{lstlisting}
for the parameter on the interior edges
and the lines~622--625 by
\begin{lstlisting}[numbers=none,language=c++,firstnumber=622]
const double a = 2.0;
const double sigma_over_h =
   2 * a * pow(k - 1, 2) /
   cell->extent_in_direction(
   GeometryInfo<dim>::unit_normal_direction[face_no]);
\end{lstlisting}
for the parameter on the boundary edges.

\subsection{Realization in Netgen/NGSolve}
The C++ software package NGSolve \cite{Schoeberl2014}
with an easy-to-use python interface
provides an implemententation of a hybridized C0IP method.
Straight-forward modifications of the existing code
enable the C0IP method presented in this paper.
%
Given an object \texttt{mesh} of the \texttt{Mesh} class
representing a regular triangulation into simplices,
the Lagrange finite element of order \(k \in \N\)
with \(2 \leq k\) can be defined as follows.
\begin{lstlisting}[numbers=none,language=python]
k = 2
fes = H1(mesh, order=k, dirichlet="left|right|bottom|top", dgjumps=True)
\end{lstlisting}
This allows to define variables for trial and test
functions.
\begin{lstlisting}[numbers=none,language=python]
u = fes.TrialFunction()
v = fes.TestFunction()
\end{lstlisting}
For an edge \(E \in \E\) with adjacent triangles
\(T_\pm \in \T\),
recall the notation
\(d_\pm \coloneqq 2 \vert T_\pm \vert / h_E\)
for the altitude of \(T_\pm\).
This quantity and the normal vector \(\nu_E\)
can be defined by special coefficient functions
in NGsolve.
\begin{lstlisting}[numbers=none,language=python]
d = specialcf.mesh_size
n = specialcf.normal(2)
\end{lstlisting}
Thus, the automatic choice of the penalty parameter
\(\sigma_{\IP, E}\) reads, for interior and boundary edges,
\begin{lstlisting}[numbers=none,language=python,firstnumber=7]
a = 4
sigma_over_h = 3 * a * k * (k - 1) / 4 * (1/d + 1/ d.Other())
sigma_over_h_bdr = 3 * a * k * (k  - 1) / 2 * 1/d
\end{lstlisting}
Abbreviating the Hessian of a function \texttt{v}
\begin{lstlisting}[numbers=none,language=python,firstnumber = 9]
def hesse(v):
    return v.Operator("hesse")
\end{lstlisting}
allows for the following variables for the average of
the binormal part of the Hessian in
the jump term \(\mathcal{J}\)
from \eqref{eq:stiffness_and_jump}
\begin{lstlisting}[numbers=none,language=python]
mean_d2udn2 = 0.5*InnerProduct(n, (hesse(u)+hesse(u.Other()))*n)
mean_d2vdn2 = 0.5*InnerProduct(n, (hesse(u)+hesse(u.Other()))*n)
\end{lstlisting}
This and the definition of the normal jump
\begin{lstlisting}[numbers=none,language=python]
def jumpdn(v):
    return n*(grad(v)-grad(v.Other()))
\end{lstlisting}
lead to the bilinear form \(A_h\) from
\eqref{eq:definition_Ah} via
\begin{lstlisting}[numbers=none,language=python]
A = BilinearForm(fes)
A += SymbolicBFI(InnerProduct (hesse(u), hesse(v)) )
A += SymbolicBFI(-mean_d2udn2 * jumpdn(v), skeleton=True)
A += SymbolicBFI(-mean_d2vdn2 * jumpdn(u), skeleton=True)
A += SymbolicBFI(sigma_over_h * jumpdn(u) * jumpdn(v),
                 VOL, skeleton=True)
A += SymbolicBFI(sigma_over_h_bdr * jumpdn(u) * jumpdn(v),
                 BND, skeleton=True)
\end{lstlisting}

\newpage

\section{Direct realization in MATLAB}
\label{app:implementation}
This subsection discusses the explicit implementation
of the lowest-order C0IP method.
The resulting software
is published online in \cite{C0IPMatlab}
under GNU general public license
and has been implemented and tested
with MATLAB version 9.9.0.1718557 (R2020b) Update 6.
The self-contained program
comprises less than 230 lines of code and
does not rely on any additional library.
Although,
the generalization to higher polynomial degrees is possible,
its implementation is much more involved.
For the sake of the general accessibility,
the presentation restricts to the lowest-order case.

\subsection{Data structures}

The implemented function solely requires
two key data structures that represent the underlying
conforming triangulation \(\T\).
The rows of the array
\texttt{n4e} \(\in \R^{\vert \V \vert \times 2}\)
contain the two coordinates of each vertex in \(\V\).
The order of the vertices in the array determines
their enumeration.
Three vertices form a triangle and their numbers
are stored in the rows of the array
\texttt{n4e} \(\in \N^{\vert \T \vert \times 3}\)
in counter-clockwise order.
If the user's data structure does not meet this convention,
the included auxiliary function \texttt{fixLocalEnumeration}
allows to reorder the indices of the vertices
in \texttt{n4e} accordingly.

All further geometric information
are computed and documented in the function
\texttt{C0IP\_geometry}.
The function \texttt{C0IP\_indices}
provides local and global enumerations
of degrees of freedom.
This minimal set of input data structures
enables the application of the C0IP realization
to triangulations from arbitrary mesh generators
such as MATLAB's \texttt{delaunay} function.

\subsection{Notation and local indices}
\label{sec:notation-indices}
Before the remaining part of this section
addresses the realization of the contributions
to the bilinear form \(A_h\)
from~\eqref{eq:definition_Ah},
the notation of the Lagrange finite element functions
and the associated local indices needs to be introduced.
Given a triangle \(T \in \T\)
with the vertices \(\V(T) = \{P_1, P_2, P_3\}\)
in counter-clockwise enumeration,
the enumeration of the midpoints \(M_1, M_2, M_3\)
of the edges \(\E(T) = \{E_1, E_2, E_3\}\)
follows the convention that
the edge \(E_j\) is opposite to the node \(P_j\)
for \(j = 1, 2, 3\)
as depicted in Figure~\ref{subfig:triangle}.
\begin{figure}
    \centering
    \subfloat[%
        Enumeration of nodes
    ]{%
        \makebox[0.40\textwidth]{%
            \label{subfig:triangle}
            \begin{tikzpicture}[%
                    font=\small,%
                    node style/.style={%
                    circle,%
                    draw,%
                    fill=white,%
                    minimum width=.8cm,%
                    inner sep=0pt%
                }
                ]
                \node[%
                    draw,%
                    regular polygon,%
                    regular polygon sides=3,%
                    minimum size=4cm%
                ]
                (T) at (0,0) {\(T\)};

                \draw (T.corner 1) node[node style] {\(P_1\)};
                \draw (T.corner 2) node[node style] {\(P_2\)};
                \draw (T.corner 3) node[node style] {\(P_3\)};
                \draw (T.side 1) node[node style] {\(M_3\)};
                \draw (T.side 2) node[node style] {\(M_1\)};
                \draw (T.side 3) node[node style] {\(M_2\)};
            \end{tikzpicture}
        }
    }
    \hfil
    \subfloat[%
        Enumeration of degrees of freedom
    ]{%
        \raisebox{1.3mm}{
            \makebox[0.40\textwidth]{%
                \label{subfig:triangle_dofs}
                \begin{tikzpicture}[%
                        font=\small,%
                        node style/.style={%
                        circle,%
                        draw,%
                        fill=white,%
                        minimum width=.5cm,%
                        inner sep=0pt%
                    }
                    ]
                    \node[%
                        draw,%
                        regular polygon,%
                        regular polygon sides=3,%
                        minimum size=4cm%
                    ]
                    (T) at (0,0) {\(T\)};

                    \draw (T.corner 1) node[node style] {\(1\)};
                    \draw (T.corner 2) node[node style] {\(2\)};
                    \draw (T.corner 3) node[node style] {\(3\)};
                    \draw (T.side 1) node[node style] {\(6\)};
                    \draw (T.side 2) node[node style] {\(4\)};
                    \draw (T.side 3) node[node style] {\(5\)};
                \end{tikzpicture}
            }
        }
    }
    \caption{%
        Local enumeration of nodes associated
        to the piecewise quadratic basis functions in the
        the triangle \(T\).
    }
\end{figure}
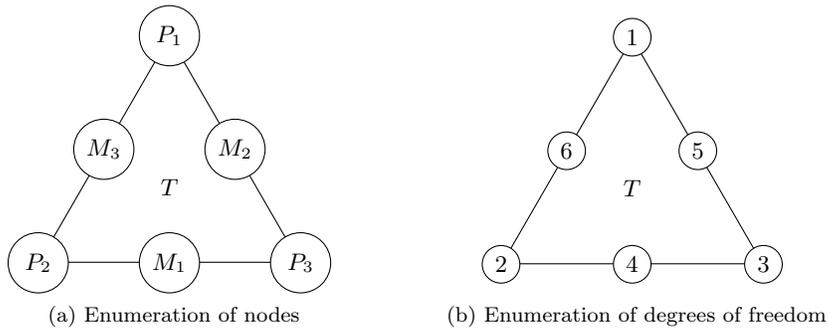
Let \(\lambda_j \in P_1(T)\) denote
the barycentric coordinate
associated to the vertex \(P_j\)
of the triangle \(T\) for \(j = 1, 2, 3\).
The shape functions
\(\varphi_1, \dots, \varphi_6\) of \(P_2(T)\)
can be represented using products
of the barycentric coordinates.
Throughout this section,
let \(\{j, k, \ell\} = \{1, 2, 3\}\)
denote any pairwise distinct indices.
The quadratic shape function
\(\varphi_j \coloneqq \lambda_j (2 \lambda_j - 1)\)
is associated to the vertex \(P_j\)
and \(\varphi_{3 + j} \coloneqq 4 \lambda_k \lambda_\ell\)
to the midpoint \(M_j\) of \(E_j\).
They satisfy the duality,
for \(j, n = 1, 2, 3\),
\begin{equation}
    \label{eq:dual_basis}
    \varphi_j(P_n)
    =
    \delta_{jn}
    =
    \varphi_{3 + j}(M_n)
    \quad
    \text{and}
    \quad
    \varphi_j(M_n)
    =
    0
    =
    \varphi_{3 + j}(P_n).
\end{equation}
The enumeration of \(\varphi_1, \dots, \varphi_6\) is illustrated in
Figure~\ref{subfig:triangle_dofs}.

Let \(E \in \E\) denote the edge
of global number \texttt{s}
\(\in \{1, \dots, \vert \E \vert \}\).
If \(E \in \E(\Omega)\) is an interior edge,
two adjacent triangles \(T_+, T_- \in \T\)
share the edge  \(\partial T_+ \cap \partial T_- = E\).
The local enumeration of the interpolation nodes
\(P^\pm_1, P^\pm_2, P^\pm_3\), \(M^\pm_1, M^\pm_2, M^\pm_3\)
in \(T_\pm\) from Figure~\ref{subfig:triangle}
determines two indices \(q, r \in \{1, 2, 3\}\)
fixed for each edge \(E\) such that
\(\Mid(E) = M^+_q = M^-_r\).
The remaining nodes are numbered accordingly
in counter-clockwise order
as displayed in Figure~\ref{subfig:edge_patch}
in a cyclic notation with index $0$ resp.\ $1$ identified
with $3$ resp.\ $4$.
\begin{figure}
    \centering
    \subfloat[%
        Enumeration in the adjacent triangles
    ]{%
        \makebox[0.40\textwidth]{%
            \label{subfig:edge_patch}
            \begin{tikzpicture}[%
                line cap=round,%
                line join=round,%
                font=\footnotesize,%
                node style/.style={%
                    circle,%
                    draw,%
                    fill=white,%
                    minimum width=.8cm,%
                    inner sep=0pt%
                },
                rotate=-120,%
            ]
                \node[%
                    draw,%
                    rotate=-120,%
                    regular polygon,%
                    regular polygon sides=3,%
                    minimum size=4cm,%
                ]
                (Tp) at (0, 0) {};
                \draw (Tp) node {\(T_+\)};

                \draw (Tp.corner 1) node[node style] {\(P^+_{q-1}\)};
                \draw (Tp.corner 2) node[node style] {\(P^+_q\)};
                \draw (Tp.corner 3) node[node style] {\(P^+_{q+1}\)};
                \draw (Tp.side 1) node[node style] {\(M^+_{q+1}\)};
                \draw (Tp.side 2) node[node style] {\(M^+_{q-1}\)};
                \draw (Tp.side 3) node[node style] {\(M^+_q\)};

                \node[%
                    draw,%
                    rotate=60,%
                    regular polygon,%
                    regular polygon sides=3,%
                    minimum size=4cm,%
                ] (Tm) at (2.8, 1.6) {};
                \draw (Tm) node {\(T_-\)};

                \draw (Tm.corner 1) node[node style] {\(P^-_{r-1}\)};
                \draw (Tm.corner 2) node[node style] {\(P^-_r\)};
                \draw (Tm.corner 3) node[node style] {\(P^-_{r+1}\)};
                \draw (Tm.side 1) node[node style] {\(M^-_{r+1}\)};
                \draw (Tm.side 2) node[node style] {\(M^-_{r-1}\)};
                \draw (Tm.side 3) node[node style] {\(M^-_r\)};

                \draw[very thick] ($(Tp.corner 3)!0.5!(Tm.corner 1)$)
                node[left] {\(E\)}
                -- ($(Tp.corner 1)!0.5!(Tm.corner 3)$);
            \end{tikzpicture}
        }
    }
    \hfil
    \subfloat[%
        Enumeration in the edge patch
    ]{%
        \label{subfig:edge_patch_dofs}
        \raisebox{0.66cm}{
            \begin{tikzpicture}[%
                line cap=round,%
                line join=round,%
                font=\small,%
                node style/.style={%
                    circle,%
                    draw,%
                    fill=white,%
                    minimum width=.7cm,%
                    inner sep=0pt%
                }
                ]
                \node[%
                draw,%
                rotate=-120,%
                regular polygon,%
                regular polygon sides=3,%
                minimum size=4cm%
                ] (T) at (0, 0) {};
                \draw (T) node {\(T_+\)};
                \draw ($(T)-(T.corner 2)$) node {\(T_-\)};

                \path ($2*(T.side 3)-(T.corner 2)$) node[coordinate] (N) {\(8\)};

                \draw (T.corner 3) -- node[node style] {\(Q_7\)} (N)
                node[node style] {\(Q_8\)}
                -- node[node style] {\(Q_9\)}
                (T.corner 1);

                \draw (T.corner 1) node[node style] {\(Q_3\)};
                \draw (T.corner 2) node[node style] {\(Q_5\)};
                \draw (T.corner 3) node[node style] {\(Q_1\)};
                \draw (T.side 1) node[node style] {\(Q_4\)};
                \draw (T.side 2) node[node style] {\(Q_6\)};
                \draw (T.side 3) node[node style] {\(Q_2\)};
            \end{tikzpicture}
        }
    }

    \caption{%
        Local enumeration of nodes
        associated to the piecewise quadratic basis functions
        in the adjacent triangles \(T_+\) and \(T_-\)
        and in the whole patch \(\omega_E\) of the edge \(E\).
        The local enumeration in each triangle
        from Figure~\ref{subfig:triangle}
        induces the choice two indices \(q, r \in \{1,  2, 3\}\)
        such that \(\Mid(E) = M^+_q = M^-_r\).
    }
    \label{fig:edge_patch}
\end{figure}
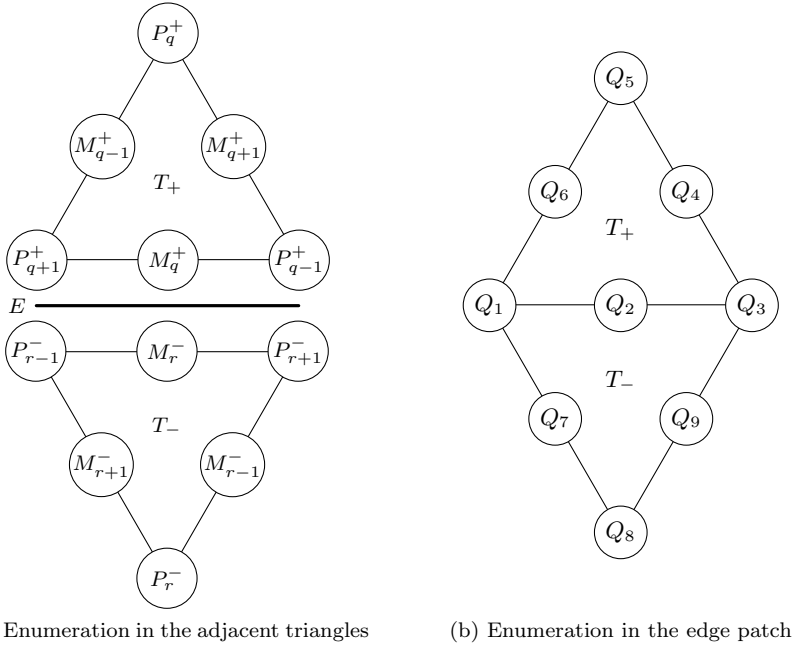
Abbreviate the nodes belonging to
the edge patch \(\omega_E\) for \(E \in \E(\Omega)\) by
\begin{equation}
    \label{eq:Simpson_nodes}
    \begin{split}
        \big[
            Q_1, Q_2, Q_3
        \big]
        &\coloneqq
        \big[
            P_{q+1}^+,
            M_q^+,
            P_{q-1}^+
        \big]
        =
        \big[
            P_{r-1}^-,
            M_r^-,
            P_{r+1}^-
        \big],\\
        \big[
            Q_4, \dots, Q_9
        \big]
        &\coloneqq
        \big[
            M^+_{q+1},
            P^+_q,
            M^+_{q-1},
            M^-_{r+1},
            P^-_r,
            M^-_{r-1}
        \big].
    \end{split}
\end{equation}
This indexing information is stored in an array
\texttt{Ind} \(\in \N^{6 \times 2}\)
containing the corresponding permutations of
\(\{1, \dots, 6\}\) as a column vector
for \(T_+\) and \(T_-\)
\begin{equation}
    \label{eq:local4s}
    \texttt{Ind}
    \coloneqq
    \begin{pmatrix}
        0 & 0 \\
        3 & 3 \\
        0 & 0 \\
        3 & 3 \\
        0 & 0 \\
        3 & 3
    \end{pmatrix}
    +
    \Psi
    \begin{pmatrix}
        q + 1 & r + 1 \\
        q     & r     \\
        q + 2 & r + 2 \\
        q + 1 & r + 1 \\
        q     & r     \\
        q + 2 & r + 2
    \end{pmatrix}
\end{equation}
where the function \(\Psi: \N_0 \to \N\)
with \(\Psi(m) \coloneqq 1 + ((m - 1) \mod 3)\)
applies component-wise to the matrix.
The nine global basis functions \(\phi_1, \dots, \phi_9 \in S^2(\T)\)
of the edge patch \(\omega_E\)
are associated to the nodes \(Q_1, \dots, Q_9\)
as in Figure~\ref{subfig:edge_patch_dofs}.
This means that
\(\phi_1, \phi_2, \phi_3\) belong to the edge \(E\),
\(\phi_4, \phi_5, \phi_6\) exclusively to \(T_+\),
and \(\phi_7, \phi_8, \phi_9\) exclusively to \(T_-\).
For \(\alpha = 1, \dots, 9\),
abbreviate the restrictions
\(\phi^\pm_\alpha \coloneqq \phi_\alpha\vert_{T_\pm} \in P_2(T_\pm)\) of $\phi_\alpha$ to $T_\pm$.

In the case of a boundary edge \(E \in \E(\partial\Omega)\),
the nodes \(Q_7, Q_8, Q_9\) corresponding to \(T_-\) are omitted
and the basis functions \(\phi_\alpha \coloneqq 0\)
vanish for \(\alpha = 7, 8, 9\).
\bigskip

%

\subsection{Discrete functions and their derivatives}
\label{sec:discrete_functions_derivatives}

This subsection presents the evaluation of
derivatives of the barycentric coordinates
and of the quadratic shape functions.
Let \(T \in \T\) denote a triangle
of global number \texttt{e}
\(\in \{1, \dots, \vert \T \vert\}\), i.e,
the global indices of its three vertices
\(\V(T) = \{P_1,  P_2, P_3\}\)
belong to the \texttt{e}-th row
of \texttt{c4n}.
Since
\(
    x
    =
    \lambda_1(x) P_1 +
    \lambda_2(x) P_2 +
    \lambda_3(x) P_3
\)
and \(1 = \lambda_1(x) + \lambda_2(x) + \lambda_3(x)\)
for all \(x \in T\),
the gradients \(\nabla\lambda_1,  \nabla \lambda_2, \nabla \lambda_3\)
are determined by the linear system of equations
\begin{equation}
    \label{eq:gradient_barycentric_coordinates}
    \begin{bmatrix}
        1 & 1 & 1\\
        P_1 & P_2 & P_3
    \end{bmatrix}
    \begin{bmatrix}
        \nabla \lambda_1^\top \\
        \nabla \lambda_2^\top \\
        \nabla \lambda_3^\top
    \end{bmatrix}
    =
    \begin{bmatrix}
        0 & 0 \\
        1 & 0 \\
        0 & 1
    \end{bmatrix}
    .
\end{equation}

\begin{lemma}
    \label{lem:gradients_barycentric_coordinates}
    Let the three-dimensional array \textup{\texttt{c4e}}
    \(\in \R^{2 \times 3 \times \vert \T \vert}\)
    contain the matrix
    \([P_1\: P_2\: P_3] \in \R^{2 \times 3}\)
    for each triangle \(T \in \T\).
    The following line of MATLAB computes
    the gradients \(\nabla\lambda_1,  \nabla \lambda_2, \nabla \lambda_3\),
    according to~\eqref{eq:gradient_barycentric_coordinates}.
\begin{lstlisting}[numbers=none]
GradP1 = [ones(1,3); c4e(:,:,e)] \ [zeros(1,2); eye(2)];
\end{lstlisting}
\end{lemma}
\begin{proof}
    The backslash operator in MATLAB calls
    the function \texttt{mldivide} and solves the
    linear system~\eqref{eq:gradient_barycentric_coordinates}
    up to machine precision.
    The three rows of the solution array
    \texttt{GradP1} \(\in \R^{3 \times 2}\)
    contain the two components of the constant gradients.
\end{proof}

From Section~\ref{sec:notation-indices},
recall the definition of the basis functions
\(\varphi_1, \dots, \varphi_6\)
associated with a triangle \(T \in \T\).
Their representation via the barycentric coordinates
and the product rule lead to the following formulas
for the derivatives of first and second order.
For pairwise distinct indices
\(\{j, k, \ell\} = \{1, 2, 3\}\),
\begin{align}
    \label{eq:gradient}
    \nabla \varphi_j
    &=
    (4 \lambda_j - 1) \nabla \lambda_j,
    &
    \nabla \varphi_{3 + j}
    &=
    4 (\lambda_k \nabla \lambda_\ell
    + \lambda_\ell \nabla \lambda_k),\\
    \label{eq:Hessian}
    \D^2 \varphi_j
    &=
    4 \nabla \lambda_j \otimes \nabla \lambda_j,
    &
    \D^2 \varphi_{3 + j}
    &=
    4\,(
        \nabla \lambda_k \otimes \nabla \lambda_\ell
        +
        \nabla \lambda_\ell \otimes \nabla \lambda_k
    ).
\end{align}
Hence, the evaluation of the gradients
\(\nabla \phi_1, \dots, \nabla \phi_6\)
in the three nodes \(P_k\), \(M_j\), and \(P_\ell\) reads,
for pairwise distinct indices
\(\{j, k, \ell\} = \{1, 2, 3\}\),
\begin{equation}
    \label{eq:gradient_evaluation}
    \begin{bmatrix}
        \nabla \varphi_j(P_j) &
        \nabla \varphi_j(P_k) \\
        \nabla \varphi_j(M_j) &
        \nabla \varphi_j(M_k) \\
        \nabla \varphi_{3 + j}(P_j) &
        \nabla \varphi_{3 + j}(P_k) \\
        \nabla \varphi_{3 + j}(M_j) &
        \nabla \varphi_{3 + j}(M_k)
    \end{bmatrix}
    =
    \begin{bmatrix}
        3 \nabla \lambda_j &
        - \nabla \lambda_j \\
        - \nabla \lambda_j &
        \nabla \lambda_j \\
        0 &
        4 \nabla \lambda_\ell \\
        2 (\nabla \lambda_k + \nabla \lambda_\ell) &
        2 \nabla \lambda_k
    \end{bmatrix}.
\end{equation}

\begin{lemma}
    \label{lem:gradient_evaluation}
    The following lines realize the evaluations
    from~\eqref{eq:gradient_evaluation}.
\begin{lstlisting}[numbers=none]
Gradient = cat(3, [3 -1 -1 -1 1 1] .* GradP1(1,:)', ...
                  [-1 3 -1 1 -1 1] .* GradP1(2,:)', ...
                  [-1 -1 3 1 1 -1] .* GradP1(3,:)', ...
                  [0 0 4 2 2 0] .* GradP1(2,:)' ...
                  + [0 4 0 2 0 2] .* GradP1(3,:)', ...
                  [0 0 4 2 2 0] .* GradP1(1,:)' ...
                  + [4 0 0 0 2 2] .* GradP1(3,:)', ...
                  [0 4 0 2 0 2] .* GradP1(1,:)' ...
                  + [4 0 0 0 2 2] .* GradP1(2,:)');
\end{lstlisting}
\end{lemma}

\begin{proof}
    Linear combinations of the gradients
    of the barycentric coordinates from \texttt{GradP1}
    provide the formulas in~\eqref{eq:gradient_evaluation}.
    They are rearranged and stored in the
    three-dimensional array \texttt{Gradient}
    \(\in \R^{2 \times 6 \times 6}\)
    containing the two components
    of the six gradients evaluated
    in the six quadrature nodes.
\end{proof}

\begin{lemma}
    \label{lem:Hessian}
    The following lines evaluate the constant
    Hessians from~\eqref{eq:Hessian}.
\begin{lstlisting}[numbers=none]
H = permute(GradP1, [2 3 1 4]) .* permute(GradP1, [3 2 4 1]);
Hessian = 4 * cat(3, H(:,:,1,1), H(:,:,2,2), H(:,:,3,3), ...
                     H(:,:,2,3) + H(:,:,3,2), ...
                     H(:,:,1,3) + H(:,:,3,1), ...
                     H(:,:,1,2) + H(:,:,2,1));
\end{lstlisting}
\end{lemma}

\begin{proof}
    All pairwise tensor products of the gradients
    of the barycentric coordinates
    from Lemma~\ref{lem:gradients_barycentric_coordinates}
    result in nine 2-by-2 matrices.
    They are stored in the four-dimensional array
    \texttt{H} \(\in \R^{2 \times 2 \times 3 \times 3}\).
    The linear combinations of these matrices
    provide the Hessian matrices
    \(\D^2 \varphi_1, \dots, \D^2 \varphi_6\)
    in the array \texttt{Hessian}
    \(\in \R^{2 \times 2 \times 6}\).
\end{proof}

\begin{remark}
    \label{rem:edge_oriented_arrays}
    The Hessian matrices
    \texttt{Hessian} \(\in \R^{2 \times 2 \times 6}\)
    from Lemma~\ref{lem:Hessian}
    and the gradient evaluations
    \texttt{Gradient} \(\in \R^{2 \times 6 \times 6}\)
    are computed elementwise and
    stored in the global variables
    \texttt{Hess4e}
    \(\in \R^{6 \times 2 \times 2 \times \vert \T \vert}\)
    and
    \texttt{Grad4e}
    \(\in \R^{6 \times 2 \times 6 \times \vert \T \vert}\)
    by
\begin{lstlisting}[numbers=none]
Hess4e(:,:,:,e) = permute(Hessian, [3 1 2]);
Grad4e(:,:,:,e) = permute(Gradient, [3 1 2]);
\end{lstlisting}
    The evaluations of the derivatives are supposed to be used
    in a parallelized computation over all edges
    for the local contributions in the
    Section~\ref{sec:jumps} below.
    MATLAB's parallel toolbox requires data structures
    to be rearranged in an edge-oriented way.
    To this end, the function \texttt{C0IP\_collectPatch}
    rearranges the arrays \texttt{Hess4e} and \texttt{Grad4e}
    to
    \[
        \mathtt{Hess4s}
        \in
        \R^{6 \times 2 \times 2 \times 2 \times \vert \E \vert}
        \quad\text{and}\quad
        \mathtt{Grad4s}
        \in
        \R^{6 \times 2 \times 6 \times 2 \times \vert \E \vert}.
    \]
    For each edge, they contain the corresponding information
    on the adjacent triangles \(T_+, T_- \in \T\)
    specified by the index in the fourth component
    (index 1 for \(T_+\) and 2
    for \(T_-\)).
    The rearrangement employs a data structure
    \texttt{e4s} whose rows contain the global indices
    of \(T_+\) and \(T_-\) for interior edges
    or the global index of \(T_+\) and \(0\) for boundary edges.
\end{remark}

\subsection{Stiffness matrix and right-hand side}
\label{sec:stiffness}

The contribution to the bilinear form \(a_\PW\)
from the triangle \(T \in \T\)
leads to the local stiffness matrix
\(A(T) \in \R^{6 \times 6}\) with,
for \(\alpha, \beta = 1,\dots, 6\),
\begin{equation}
    \label{eq:local_stiffness}
    A_{\alpha\beta}(T)
    \coloneqq
    \int_T \D^2 \varphi_\alpha : \D^2 \varphi_\beta \;\mathrm{d}x
    =
    \vert T \vert\, (\D^2 \varphi_\alpha : \D^2 \varphi_\beta).
\end{equation}
\begin{lemma}
    Let the array
    \textup{\texttt{area4e}} \(\in \R^{\vert \T \vert}\)
    contain the area \(\vert T \vert\)
    of each triangle \(T \in \T\)
    and \textup{\texttt{Hessian}}
    \(\in \R^{2 \times 2 \times 6}\)
    the evaluation of the Hessians of the six basis
    functions from Lemma~\ref{lem:Hessian}.
    The following lines
    compute~\eqref{eq:local_stiffness}.
\begin{lstlisting}[numbers=none]
A4e(:,:,e) = area4e(e) * sum(permute(Hessian, [3 4 1 2]) ...
                             .* permute(Hessian, [4 3 1 2]), [3 4]);
\end{lstlisting}
\end{lemma}
\begin{proof}
    The pairwise Frobenius scalar products
    of the Hessian matrices
    \(\D^2 \varphi_\alpha\), \(\alpha = 1,\dots,6\)
    from the array \texttt{Hessian}
    provide the array
    \texttt{Integrand} \(\in \R^{6 \times 6}\).
    The multiplication with the area
    completes the computation of~\eqref{eq:local_stiffness}.
\end{proof}

The local contribution \(b(T) \in \R^6\)
from the triangle \(T \in \T\) to the right-hand side
consists of approximations to the integral
\(\int_T f \varphi_\alpha \,\mathrm{d} x\) for
\(\alpha = 1, \dots, 6\).
Its computation employs
the edge-oriented midpoint quadrature rule
and the dual basis property~\eqref{eq:dual_basis}
\begin{equation}
    \label{eq:local_RHS}
    b_\alpha(T)
    \coloneqq
    \frac{\vert T \vert}{3}
    \sum_{\ell = 1}^3
    f(M_\ell) \varphi_\alpha(M_\ell)
    =
    \begin{cases}
        0 & \text{for } \alpha = 1, 2, 3,\\
        \displaystyle
        \frac{\vert T \vert}{3}
        f(M_{\alpha - 3})
        & \text{for } \alpha = 4, 5, 6.
    \end{cases}
\end{equation}
\begin{lemma}
    Let the rows of the array
    \textup{\texttt{n4s}}
    \(\in \N^{\vert \E \vert \times 2}\)
    specify the global numbers
    of the two endpoints of each edge.
    Its order determines an enumeration of the edges
    of the triangulation and the numbers
    of the three edges of each triangle
    form the rows of an array
    \textup{\texttt{s4e}}
    \(\in \N^{\vert \T \vert \times 3}\).
    Suppose that the array
    \textup{\texttt{mid4s}}
    \(\in \R^{\vert \E \vert \times 2}\)
    contains the coordinates
    of the midpoints of the edges.
    Given a function handle \texttt{f}
    describing the right-hand side
    \(f: \overline\Omega \to \R\),
    the evaluations of \(f\) required for~\eqref{eq:local_RHS}
    read
\begin{lstlisting}[numbers=none]
f4mid = f(mid4s);
valF4e = f4mid(s4e);
\end{lstlisting}
    On the triangle \(T\) with index \textup{\texttt{e}},
    the following line computes~\eqref{eq:local_RHS}.
\begin{lstlisting}[numbers=none]
b4e(:,e) = [zeros(1, 3), area4e(e) / 3 * valF4e(e,:)];
\end{lstlisting}
\end{lemma}

\begin{proof}
    The right-hand side \(f\) is evaluated
    in the midpoints of each edge first.
    The array \texttt{s4e} allows
    to select the three point evaluations
    of \(f\) on each triangle into the array
    \texttt{valF4e} \(\in \R^{\vert \T \vert \times 3}\).
    On the triangle \(T\),
    the multiplication with the quadrature weigths
    completes the computation of~\eqref{eq:local_RHS}.
\end{proof}

\subsection{Jump term} \label{sec:jumps}

This subsection presents the MATLAB realization
of averages and jumps of the derivatives
from Section~\ref{sec:discrete_functions_derivatives}
as well as the local contributions
to the bilinear form \(\mathcal{J}\).

Recall the definition of the basis functions
\(\phi_1, \dots, \varphi_9\)
associated with an edge patch \(\omega_E\)
for \(E \in \E\)
from Subsection~\ref{sec:notation-indices}.
The evaluations of the normal jumps
of the gradients \(\nabla \phi_1, \dots, \nabla \phi_9\)
in the nodes \(Q_1, Q_2, Q_3\) on the edge \(E\)
read,
for \(\alpha = 1, \dots, 9\)
and \(\ell = 1, 2, 3\),
\begin{equation}
    \label{eq:evaluation_normal_jump}
    [\nabla \phi_\alpha (Q_\ell) \cdot \nu_E]_E
    =
    \begin{cases}
        \big(
        \nabla \phi^+_\alpha (Q_\ell)
        - \nabla \phi^-_\alpha (Q_\ell)
        \big) \cdot \nu_E
        &\quad\text{if } E \in \E(\Omega),\\
        \nabla \phi^+_\alpha (Q_\ell) \cdot \nu_E
        &\quad\text{if } E \in \E(\partial\Omega),
    \end{cases}
\end{equation}

\begin{lemma}
	\label{lem:computation_jumpnormal}
	Let the rows of the array \textup{\texttt{normal4s}}
	\(\in \R^{\vert \E \vert \times 2}\)
	contain the two components of the unit normal
	vector \(\nu_E \in \R^2\) for each edge \(E\).
	Assume that its sign is chosen such that
	\(\nu_E = \nu_{T_+}\) for the outward unit normal vector
	\(\nu_{T_+}\) of the triangle \(T_+ \in \T\).
	Suppose that the five-dimensional array
	\textup{\texttt{Grad4s}}
	\(\in \R^{6 \times 2 \times 6 \times 2 \times \vert \E \vert}\)
	contains the evaluations of the
	gradients of the six basis functions
	\(\nabla \phi^\pm_\alpha\),
	\(\alpha = 1, \dots, 6\),
	in the six nodes on \(T_\pm\)
	for each edge \(E \in \E\).
	Recall the index array
	\textup{\texttt{Ind}} \(\in \N^{6 \times 2}\)
	from~\eqref{eq:local4s}
	The evaluation of the normal jumps
	from~\eqref{eq:evaluation_normal_jump} reads
\begin{lstlisting}[numbers=none,basicstyle=\upshape\ttfamily\scriptsize]
GradNormal = permute(sum(Grad4s(:,:,:,:,s) .* normal4s(s,:), 2), [1 3 4 2]);
JumpNormal = [GradNormal(Ind(1:3,1),Ind(1:3,1),1) ...
              - GradNormal(Ind(3:-1:1,2),Ind(3:-1:1,2),2); ...
              GradNormal(Ind(4:6,1),Ind(1:3,1),1); ...
              - GradNormal(Ind(4:6,2),Ind(3:-1:1,2),2)];
\end{lstlisting}
\end{lemma}

\begin{proof}
	The array \texttt{GradNormal}
	\(\in \R^{6 \times 6 \times 2}\)
	contains the evaluations of the normal derivatives
	\(\nabla \phi_\alpha^\pm \cdot \nu_E\) for \(\alpha = 1,\dots, 6\)
	at the nodes \(P_1^\pm, P_2^\pm, P_3^\pm\),
	\(M_1^\pm, M_2^\pm, M_3^\pm\) on \(T_\pm\).
	Their linear combination lead to the jumps
	along the edge in the array \texttt{JumpNormal}
	\(\in \R^{9 \times 3}\)
	for the basis functions \(\phi_\alpha\)
	associated to the nodes \(Q_\alpha\)
	for \(\alpha = 1,\dots,9\).
	The first three basis functions \(\phi_1, \phi_2, \phi_3\)
	belong to the edge \(E\) and the jump
	consists of the difference of contributions
	from \(T_+\) and \(T_-\).
	The array \texttt{Ind} selects the corresponding
	basis functions on the adjacent triangles
	as well as the corresponding nodes \(Q_1, Q_2, Q_3\)
	belonging to the edge \(E\).
	Note that \texttt{GradNormal} vanishes on \(T_-\)
	for boundary edges.
	The basis functions \(\phi_4, \phi_5, \phi_6\)
	vanish on \(T_-\) and
	the jump of their normal derivatives from \(T_+\)
	is computed with 0.
	Vice versa, the basis functions \(\phi_7, \phi_8, \phi_9\)
	lead to the jump of 0 with the normal derivatives from \(T_-\).
\end{proof}

The averages of the piecewise Hessian matrices
from \eqref{eq:Hessian} read,
for \(\alpha = 1, \dots, 9\),
\begin{equation}
    \label{eq:Hessian_average}
    \langle (\D^2_\PW \phi_\alpha \nu_E) \cdot \nu_E \rangle_E
    =
    \begin{cases}
        ((\D^2 \phi^+_\alpha + \D^2 \phi^-_\alpha)
        \nu_E / 2) \cdot \nu_E
        &\quad\text{if } E \in \E(\Omega),\\
        (\D^2 \phi^+_\alpha \, \nu_E) \cdot \nu_E
        &\quad\text{if } E \in \E(\partial\Omega).
    \end{cases}
\end{equation}

\begin{lemma}
    \label{lem:Hessian_average}
    The following lines compute the average~\eqref{eq:Hessian_average}.
\begin{lstlisting}[numbers=none]
HessBinormal = ...
    permute(sum(Hess4s(:,:,:,:,s) .* normal4s(s,:) ...
            .* permute(normal4s(s,:), [1 3 2]), [2 3]), [1 4 2 3]);
MeanHessBinormal = ...
    [HessBinormal(Ind(1:3,1),1) + HessBinormal(Ind(3:-1:1,2),2);...
     HessBinormal(Ind(4:6,1),1); HessBinormal(Ind(4:6,2),2)];
if isInteriorSide(s); MeanHessBinormal = MeanHessBinormal / 2; end
\end{lstlisting}
\end{lemma}

\begin{proof}
    The array \texttt{HessBinormal}
    \(\in \R^{6 \times 2}\)
    contains the scalar product
    \(
        (\D^2 \phi^\pm_\alpha \, \nu_E) \cdot \nu_E
    \)
    of each Hessian matrix
    \(\D^2 \phi^\pm_\alpha\) for \(\alpha = 1, \dots, 6\)
    on \(T_+\) and \(T_-\)
    with the unit normal vector \(\nu_E\).
    The linear combination for the averages
    \(\langle (\D^2_\PW \phi_\alpha \nu_E) \cdot \nu_E \rangle_E\)
    along \(E\) for \(\alpha = 1, \dots, 9\)
    is computed analogously to \texttt{JumpNormal}
    in the proof of Lemma~\ref{lem:computation_jumpnormal}
    with constant Hessian matrices replacing
    the gradient evaluations.
    The averages are stored in the array
    \texttt{MeanHessBinormal} \(\in \R^9\).
    Finally, the prefactor \(1/2\) is applied for
    interior edges only according to~\eqref{eq:Hessian_average}.
\end{proof}

The contribution to the bilinear form \(\mathcal{J}\)
from the edge \(E \in \E\)
consists of the local matrix \(J(E) \in \R^{9 \times 9}\) with,
for \(\alpha, \beta = 1,\dots, 9\),
\begin{equation}
    \label{eq:local_jump_term}
    \begin{split}
        J_{\alpha\beta}(E)
        &\coloneqq
        \int_E
        \langle (\D^2_\PW \phi_\alpha \nu_E) \cdot \nu_E \rangle_E
        \cdot
        [\nabla \phi_\beta \cdot \nu_E]_E
        \;\mathrm{d}s
        \\
        &=
        \vert E \vert\,
        \langle (\D^2_\PW \phi_\alpha \nu_E) \cdot \nu_E \rangle_E
        \cdot
        [\nabla \phi_\beta (Q_2) \cdot \nu_E]_E.
    \end{split}
\end{equation}
If \(E \in \E(\partial\Omega)\),
\(\phi_7, \phi_8, \phi_9 \equiv 0\)
implies that
\(J_{\alpha\beta}(E) = 0\) vanishes for all \(\alpha, \beta = 1,\dots,9\)
with \(7 \leq \alpha\) or \(7 \leq \beta\).

\begin{lemma}
    Let the array \textup{\texttt{length4s}}
    \(\in \R^{\vert \E \vert}\)
    contain the length of each edge.
    The following line computes~\eqref{eq:local_jump_term}.
\begin{lstlisting}[numbers=none]
J4s(:,:,s) = length4s(s) * (MeanHessBinormal * JumpNormal(:,2)');
\end{lstlisting}
\end{lemma}

\begin{proof}
    The pairwise scalar products for
    \(\alpha, \beta = 1,\dots, 9\)
    of the evaluations from the
    Lemmas~\ref{lem:computation_jumpnormal}
    and~\ref{lem:Hessian_average}
    followed by the multiplication
    with the length of the edge \(E\)
    realize to the computation of
    the integral in~\eqref{eq:local_jump_term}.
\end{proof}

\subsection{Penalty term}
\label{sec:penalty}
In this subsection,
the automatic choice of the penalty parameter
\(\sigma_{\IP,E}\) from~\eqref{eq:definition_sigma}
is the point of departure.
\begin{lemma}
	\label{lem:computation_sigma}
	Let the vector \textup{\texttt{area4e}}
	\(\in \R^{\vert \T \vert}\)
	provide the area of each triangle
	and \textup{\texttt{length4s}}
	\(\in \R^{\vert \E \vert}\)
	the length of each edge.
	Suppose the rows \textup{\texttt{e4s}}
	\(\in \N_0^{\vert \E \vert \times 2}\)
	contain the global indices of \(T_+\) and \(T_- \in \T\)
	for interior edges
	\(E = \partial T_+ \cap \partial T_- \in \E(\Omega)\)
	or the global index of \(T_+ \in \T\) and \(0\)
	for boundary edges
	\(E \in \E(T_+) \cap \E(\partial\Omega)\).
	Let the boolean array
	\textup{\texttt{isInteriorSide}}
	\(\in \{0, 1\}^{\vert \E \vert}\)
	determine whether an edge belongs to the
	interior of the domain.
	Given a parameter \textup{\texttt{a}}
	\( \coloneqq a > 1\),
	the following lines compute
	the automatic penalization parameter \(\sigma_{\IP,E}\)
	from~\eqref{eq:definition_sigma}.
\begin{lstlisting}[numbers=none]
sigma4s = 3 * a * (length4s.^2) ./ area4e(Indices.e4s(:,1));
sigma4s(isInteriorSide) = sigma4s(isInteriorSide) / 4 ...
                          + 3/4 * a * length4s(isInteriorSide).^2 ...
                            ./ area4e(Indices.e4s(isInteriorSide,2));
\end{lstlisting}
\end{lemma}

\begin{proof}
	The parameter \(\sigma_{\IP,E}\)
	consists of the sum of one contribution
	from each adjacent triangle \(T_+\) and \(T_-\).
	The value for all edges \(E \in \E\) is initialized
	by \(3a\,  h_E^2 / \vert T_+ \vert\).
	In case of an interior edge,
	the contribution from \(T_+\)
	is multiplied with the prefactor \(1/4\)
	and the term \(3a\, h_E^2 / (4\vert T_- \vert)\)
	for \(T_-\) is added.
\end{proof}

The contribution to the penalty term \(c_\IP\)
from the edge \(E \in \E\)
leads to the local matrix \(C(E) \in \R^{9 \times 9}\) with,
for \(\alpha, \beta = 1, \dots, 9\),
\[
C_{\alpha\beta}(E)
\coloneqq
\frac{\sigma_{\IP,E}}{h_E}
\int_E
[\nabla \phi_\alpha \cdot \nu_E]_E\,
[\nabla \phi_\beta \cdot \nu_E]_E
\;\mathrm{d}s.
\]
The three-point one-dimensional Simpson quadrature rule
with quadrature nodes \(Q_1,  Q_2, Q_3\)
from~\eqref{eq:Simpson_nodes}
and weights \(w_1 = w_3 = 1/6\) and \(w_2 = 2/3\)
integrates the product of the jump terms exactly and
reads
\begin{equation}
	\label{eq:local_penalty_term}
	C_{\alpha\beta}(E)
	=
	\sigma_{\IP,E}
	\sum_{j = 1}^3
	w_j\,
	[\nabla \phi_\alpha(Q_j) \cdot \nu_E]_E \,
	[\nabla \phi_\beta(Q_j) \cdot \nu_E]_E.
\end{equation}
If \(E \in \E(\partial\Omega)\),
\(\phi_7, \phi_8, \phi_9 \equiv 0\)
implies that
\(C_{\alpha\beta}(E) = 0\) vanishes for all \(\alpha, \beta = 1,\dots,9\)
with \(7 \leq \alpha\) or \(7 \leq \beta\).

\begin{lemma}
	\label{lem:computation_penalty_term}
        Suppose that \textup{\texttt{JumpNormal}}
        \(\in \R^{9 \times 3}\) contains the evaluation
        of the normal jumps from Lemma~\ref{lem:computation_jumpnormal}.
	The realization of the quadrature rule
	in~\eqref{eq:local_penalty_term} reads
	\begin{lstlisting}[numbers=none,basicstyle=\upshape\ttfamily\scriptsize]
Integrand = permute(JumpNormal, [1 3 2]) .* permute(JumpNormal, [3 1 2]);
C4s(:,:,s) = sigma4s(s)/6 * sum(cat(3, 1, 4, 1) .* Integrand, 3);
	\end{lstlisting}
\end{lemma}

\begin{proof}
	The pairwise products of the evaluations
	of the normal jumps
	\([\nabla \phi_\alpha \cdot \nu_E]_E\),
	\(\alpha = 1, \dots, 9\),
	in \texttt{JumpNormal}
	form the array \texttt{Integrand}
	\(\in \R^{9 \times 3 \times 9}\).
	The multiplication with the penalty parameter
	\(\sigma_{\IP,E}\) and the quadrature weights
	followed by the sum over the three terms
	of the quadrature rule
	complete the computation of~\eqref{eq:local_penalty_term}.
\end{proof}

\subsection{Assembling} \label{sec:assembling}

This section describes the assembling of the global
system matrix using the local contributions
from the Sections~\ref{sec:stiffness}--\ref{sec:penalty}.
To this end,
let \(\V = \{z_1, \dots, z_K\}\) denote
a global enumeration of
the \(K \coloneqq \vert \V \vert \in \N\) vertices
and \(\E = \{E_1, \dots, E_L\}\) of
the \(L \coloneqq \vert \E \vert \in \N\) edges
of the triangulation \(\T\).
The edge midpoints
\(\Mid(\E) = \{m_\ell \coloneqq \Mid(E_\ell) \;:\; \ell = 1,\dots, L\}\)
are numbered accordingly.
This induces a global enumeration of the
\(\mathtt{N} \coloneqq N \coloneqq K + L\) basis functions
\(\varphi_1, \dots, \varphi_N \in S^2(\T)\) via,
for all \(j,k = 1, \dots, K\) and \(\ell,n = 1, \dots, L\),
\[
    \varphi_j(z_k) = \delta_{jk},
    \quad
    \varphi_{K+\ell}(m_n) = \delta_{\ell n},
    \quad\text{and}\quad
    \varphi_j(m_n) = 0 = \varphi_{K+\ell}(z_k).
\]

The bilinear form \(a_\PW\)
leads to the global stiffness matrix \(A \in \R^{N \times N}\)
with
\[
    A_{jk}
    \coloneqq
    a_\PW(\varphi_j, \varphi_k)
    =
    \sum_{T \in \T}
    \int_T
    \D^2 \varphi_j : \D^2 \varphi_k
    \;\mathrm{d}x \quad \text{for all } j,k = 1, \dots, N.
\]
The integration of the right-hand side
\(\int_\Omega f \, \varphi_j \dx\)
from~\eqref{eq:discrete_form}
employs the edge-oriented midpoint quadrature rule.
The resulting right-hand side vector \(b \in \R^N\) reads,
for all \(j = 1, \dots, K\) and \(\ell = 1,\dots, L\),
\[
    b_j
    \coloneqq
    0
    \quad\text{and}\quad
    b_{K + \ell}
    \coloneqq
    \sum_{T \in \T, E_\ell \subset T}
    \frac{\vert T \vert}{3}
    f(\Mid(E_\ell)).
\]
The assembling of \(A\) and \(b\)
with the local contributions \(A(T)\) and \(b(T)\)
from Section~\ref{sec:stiffness}
requires the translation of
local indices of shape functions on the triangle \(T \in \T\)
into global indices of basis functions.
The index set
\begin{align*}
    I(T)
    &\coloneqq
    \big\{
    (\alpha, j) \in \{1,2,3\} \times \{1, \dots, K\}
    \;:\;
    P_\alpha = z_j
    \big\}\\
    &\hphantom{{}\coloneqq{}}
    \cup
    \big\{
    (3 + \alpha, K + \ell)
    \;:\;
    (\alpha, \ell) \in \{1,2,3\} \times \{1, \dots, L\}
    \text{ with }
    M_\alpha = m_\ell
    \big\}.
\end{align*}
is stored in the variable
\texttt{global4e = [n4e, nNodes + s4e]'}
in that the index \(j\) has position \(\alpha\)
in the column associated to \(T\)
and analogously for \(K+\ell\) and \(3 + \alpha\).
The set \(I(T)\) leads to the assembling of
\(A \in \R^{N \times N}\) and \(b \in \R^N\)
via
\begin{align}
    \label{eq:assembling_A}
    A
    &=
    \sum_{T \in \T}
    \sum_{(\alpha, j) \in I(T)}
    \sum_{(\beta, k) \in I(T)}
    A_{\alpha\beta}(T) \, e_j \otimes e_k,
    \\
    \label{eq:assembling_b}
    b
    &=
    \sum_{T \in \T}
    \sum_{(\alpha, j) \in I(T)}
    b_\alpha(T) \, e_j.
\end{align}

\begin{lemma}
    \label{lem:triangle_assembling}
    The following lines compute \(A\) and \(b\)
    from~\eqref{eq:assembling_A}--\eqref{eq:assembling_b}.
\begin{lstlisting}[numbers=none]
RowIndices4e = permute(repmat(global4e, [1 1 6]), [1 3 2]);
ColIndices4e = permute(RowIndices4e, [2 1 3]);
A = sparse(RowIndices4e(:), ColIndices4e(:), A4e(:), N, N);
b = accumarray(dof4e(:), b4e(:), [N 1]);
\end{lstlisting}
\end{lemma}

\begin{proof}
    The indices of \(I(T)\) in \texttt{global4e}
    are rearranged to match the requirements
    of MATLAB's \texttt{sparse} command.
    This results in the two arrays
    \texttt{RowIndices4e}, \texttt{ColIndices4e}
    \(\in \N^{6 \times 6 \times \vert \T \vert}\)
    of the same size as \texttt{A4e}.
    For each entry in \texttt{A4e},
    the corresponding entries in
    \texttt{RowIndices4e} and \texttt{ColIndices4e}
    provide the row and column indices of the global
    position in \texttt{A}.
    The \texttt{accumarray} function realizes the
    same procedure for dense one-dimensional
    arrays.
\end{proof}

The bilinear forms \(\mathcal{J}\) and \(c_\IP\)
lead to the global matrices
\(J \in \R^{N \times N}\)
and \(C \in \R^{N \times N}\)
with, for all \(j,k = 1, \dots, N\),
\begin{align*}
    J_{jk}
    &\coloneqq
    \mathcal{J}(\varphi_j, \varphi_k)
    =
    \sum_{E \in \E}
    \int_E
    \langle (\D^2_\PW \varphi_j \nu_E) \cdot \nu_E \rangle_E
    \cdot
    [\nabla \varphi_k \cdot \nu_E]_E
    \;\mathrm{d}s,\\
    C_{jk}
    &\coloneqq
    c_\IP(\varphi_j, \varphi_k)
    =
    \sum_{E \in \E}
    \frac{\sigma_{\IP,E}}{h_E}
    \int_E
    [\nabla \varphi_j \cdot \nu_E]_E\,
    [\nabla \varphi_k \cdot \nu_E]_E
    \;\mathrm{d}s.
\end{align*}
Recall the enumeration of the nodes \(Q_1, \dots, Q_9\)
from~\eqref{eq:Simpson_nodes}
as displayed in Figure~\ref{fig:edge_patch}.
The translation of
local indices of shape functions on the edge patch \(\omega_E\)
into global indices of basis functions
employs the index sets
\(I(E) \coloneqq I^+(E) \cup I^-(E)\)
defined by
\begin{align*}
    I^+(E)
    &\coloneqq
    \big\{
        (\alpha, j)
        \in \{1,3,5\} \times \{1, \dots, K\}
        \;:\;
        Q_\alpha = z_j
    \big\}
    \\
    &\hphantom{{}\coloneqq{}}
    \cup
    \big\{
        (\alpha, M + \ell)
        \;:\;
        \alpha \in \{2,4,6\}, \ell \in \{1, \dots, L\},
        \:
        Q_\alpha = m_\ell
    \big\}
\end{align*}
and \(I^-(E)\) according to the location of \(E\).
For an interior edge \(E \in \E(\Omega)\),
the indices from \(T_-\) read
\begin{align*}
    I^-(E)
    &\coloneqq
    \big\{
    (8, j) \;:\; j \in \{1, \dots, K\},\: Q_8 = z_j
    \big\}
    \\
    &\hphantom{{}\coloneqq{}}
    \cup
    \big\{
    (\alpha, K + \ell)
    \;:\;
    \alpha \in \{7, 9\},
    \ell \in \{1, \dots, K\},\: Q_\alpha = m_\ell
    \big\}.
\end{align*}
For a boundary edge \(E \in \E(\partial\Omega)\),
recall that the local contributions
\(J_{\alpha\beta}(E) = C_{\alpha\beta}(E) = 0\)
from the Sections~\ref{sec:jumps} and~\ref{sec:penalty}
vanish for all \(\alpha, \beta = 1,\dots,9\)
with \(7 \leq \alpha\) or \(7 \leq \beta\).
The vanishing entries do not influence the assembling
of the global matrix and
the corresponding index pairs can be chosen arbitrarily,
e.g.,
\(
    I^-(E)
    \coloneqq
    \{(7,1), (8,1), (9,1)\}
\).
The set \(I(E)\) is stored in a variable
\texttt{dof4s} \(\in \N^{\vert \E \vert \times 9}\)
similarly to \texttt{dof4e} for \(I(T)\).
The assembling of \(J \in \R^{N \times N}\)
with the local contributions \(J(E)\)
from Section~\ref{sec:jumps} reads
\begin{equation}
    \label{eq:edge_assembling}
    J
    =
    \sum_{E \in \E}
    \sum_{(\alpha, j) \in I(E)}
    \sum_{(\beta, k) \in I(E)}
    J_{\alpha\beta}(E) \, e_j \otimes e_k,
\end{equation}
and analogously for
\(C \in \R^{N \times N}\)
with \(C_{\alpha\beta}(E)\)
from Section~\ref{sec:penalty}
replacing \(J_{\alpha\beta}(E)\).
The MATLAB realization of~\eqref{eq:edge_assembling}
is verbatim to Lemma~\ref{lem:triangle_assembling}
with \texttt{dof4s} replacing \texttt{dof4e}.

The sum of the three contributions results in the
system matrix
\[
    B
    \coloneqq
    A - J - J^\top + C
    \in
    \R^{N \times N}.
\]
Since the discrete solution \(u_\IP \in S^2_0(\T)\)
to~\eqref{eq:discrete_form}
vanishes along the boundary,
the indices of the degrees of freedom in the set
\[
    I(\Omega)
    \coloneqq
    \big\{
        j \in \{1, \dots, M\}
        \;:\;
        z_j \in \V(\Omega)
    \big\}
    \cup
    \big\{
        M + \ell
        \;:\;
        \ell \in \{1, \dots, L\},\:
        E_\ell \in \E(\Omega)
    \big\}
\]
restrict to all interpolation nodes
in the interior of the domain.
The discrete formulation~\eqref{eq:discrete_form}
seeks the coefficient vector
\(x \in \R^{M + L}\) of
\[
    u_\IP
    =
    \sum_{j = 1}^{M + L}
    x_j \varphi_j
    \in
    S^2_0(\T)
\]
with \(x_j = 0\) for all boundary indices
\(j \in \{1, \dots, M + L\} \setminus I(\Omega)\)
satisfying
the linear system of equations
\[
    (B_{jk})_{j,k \in I(\Omega)} (x_k)_{k \in I(\Omega)}
    =
    (b_j)_{j \in I(\Omega)}.
\]
The solution of this system concludes the realization
of the C\textsuperscript{0} interior penalty
method from~\eqref{eq:discrete_form}.

The practical implementation of
the eigenvalue problem~\eqref{eq:EVP}
employs the matrices \texttt{A}, \texttt{B}, and \texttt{C}
and MATLAB's \texttt{eigs} function
in one line.
\begin{lstlisting}[numbers=none]
[~, lambda] = eigs(B(dof,dof), A(dof,dof) + C(dof,dof), 1, 0);
\end{lstlisting}

\subsection{A posteriori error estimation}
\label{app:estimator}
The estimator from~\eqref{eq:estimator}
can be evaluated using some of the data structures above.
\begin{lstlisting}[numbers=none]
% Compute volume term
vol4e = area4e .* sum(valF4e.^2, 2) / 3;
% Local coefficient vectors
u4e = u(global4e');
u4s = u(global4s');
% Binormal Hessian jump term
HessU4e = squeeze(sum(Hess4e .* permute(u4e, [2 3 4 1]), 1));
signedNormal4e = ...
    permute(reshape(normal4s(s4e(:),:), [nElem 3 2]), [3 2 1]);
HessBinormalU4e = ...
    squeeze(sum(permute(normal4e, [1 4 3 2]) .* HessU4e ...
                .* permute(signedNormal4e, [4 1 3 2]), [1 2]));
jump4s = accumarray(s4e(:), HessBinormalU4e(:), [nSides 1]);
jump4s = length4s.^2 .* jump4s.^2;
% Remove boundary contributions
jump4s(CbSides) = 0;
% Compute penalty term
penalty4s = sigma4s .* squeeze(sum(C4s .* permute(u4s, [2 3 1]) ...
                                    .* permute(u4s, [3 2 1]), [1 2]));
% Assembling
eta4e = (area4e.^2) .* vol4e + sum(jump4s(s4e), 2)...
        + sum(penalty4s(s4e), 2);
eta = sqrt(sum(eta4e));
\end{lstlisting}

\end{document}